\documentclass[a4paper,11pt,oneside]{article}
\usepackage{latexsym,amsfonts,amsmath,amssymb,amsthm}
\usepackage{tkz-berge}

\newcommand{\Z}{\mathbb{Z}}

\newcommand{\pres}[2]{\langle {#1}\ |\ {#2} \rangle}

\newtheorem{theorem}{Theorem}

\newtheorem{lemma}[theorem]{Lemma}

\newtheorem{corollary}[theorem]{Corollary}
\newtheorem{maintheorem}{Theorem}

\numberwithin{theorem}{section}

\theoremstyle{definition}
\newtheorem{defn}[theorem]{Definition}
\newtheorem{remark}[theorem]{Remark}
\newtheorem{example}[theorem]{Example}

\begin{document}
\title{Generalized polygons and star graphs of cyclic presentations of groups}%
\author{Ihechukwu Chinyere and Gerald Williams\thanks{This work was supported by the Leverhulme Trust Research Project Grant RPG-2017-334.}}

\maketitle

\begin{abstract}
Groups defined by presentations for which the components of the corresponding star graph are the incidence graphs of generalized polygons are of interest as they are small cancellation groups that -- via results of Edjvet and Vdovina --  are fundamental groups of polyhedra with the generalized polygons as links and so act on Euclidean or hyperbolic buildings; in the hyperbolic case the groups are SQ-universal. A cyclic presentation of a group is a presentation with an equal number of generators and relators that admits a particular cyclic symmetry. We obtain a classification of the non-redundant cyclic presentations where the components of the corresponding star graph are generalized polygons. The classification reveals that both connected and disconnected star graphs are possible and that only generalized triangles (i.e.\,incidence graphs of projective planes) and regular complete bipartite graphs arise as the components. We list the presentations that arise in the Euclidean case and show that at most two of the corresponding groups are not SQ-universal (one of which is not SQ-universal, the other is unresolved). We obtain results that show that many of the SQ-universal groups are large.
\end{abstract}


\noindent \textbf{MSCs:} 51E24, 05E18, 20F67, 20F05, 52B05, 57M07.

\section{Introduction}\label{sec:intro}

An $(m,k)$-special presentation is a group presentation in which the relators have length $k$ and whose star graph is the incidence graph of a generalized polygon~\cite{EdjvetVdovina}. (Definitions will be given in Section~\ref{sec:prelims}.) This generalizes the concept of special presentations (from~\cite{Howie89}) which corresponds to the case $(m,k)=(3,3)$, and so the star graph is a generalized triangle (or the incidence graph of a projective plane). Polygonal presentations (for pairs of natural numbers $\nu,k$) were introduced in~\cite{Vdovina02} (see also~\cite{Vdovina05}) as a tool for constructing polyhedra with specified links. The concept was generalized in~\cite{EdjvetVdovina}, where it was shown that a polyhedron (obtained by identifying edges of a set of $k$-gons) has $\nu$ vertices with links $\Gamma_0,\ldots ,\Gamma_{\nu-1}$ if and only if the polyhedron corresponds to a polygonal presentation for $\nu,k$ over these links and that, given a non-redundant $(m,k)$-special presentation with star graph $\Gamma$, there exists a polygonal presentation over $\Gamma$ whose corresponding polyhedron has $\Gamma$ as its link.

The $(3,3)$-special presentations whose star graph is the smallest generalized triangle (the Heawood graph) were classified in~\cite{EdjvetHowie}. It was shown in~\cite{Howie89} that for any prime power $q-1$ there is a $(3,3)$-special presentation whose star graph is the incidence graph of the Desarguesian projective plane over the Galois field of order $q-1$, and an example machine for their construction was provided. Polygonal presentations for $k=3$, $\nu=1$ and where $\Gamma$ is the smallest or second smallest generalized triangle were classified in~\cite{CMSZ1,CMSZ2} (where they are called triangular presentations). An example of a polygonal presentation for $k=3,\nu=1$ that corresponds to a non-Desarguesian projective plane (the Hughes plane) was given in~\cite{Radu}.
Polygonal presentations for $k=3,\nu=1$ and where $\Gamma$ is the smallest generalized quadrangle were obtained in~\cite{KangaslampiVdovina06}, and all such polygonal presentations were classified in~\cite{KangaslampiVdovinaIJAC,CarboneKangaslampiVdovina} (subgroups of the groups acting on the corresponding polyhedron were studied in \cite{KangaslampiVdovina17}). Burger-Mozes presentations furnish examples of $(2,4)$-special presentations. SQ-universality of groups defined by special presentations was studied in~\cite{EdjvetVdovina}, where the problem of when such groups are large was posed~\cite[Problem~2]{EdjvetVdovina}. In this article we generalize the concept of $(m,k)$-special presentations to $(m,k,\nu)$-special presentations by replacing the condition that the star graph is a generalized $m$-gon with the condition that it has $\nu\geq 1$ isomorphic components, each of which is a generalized $m$-gon. We give the first examples of $(m,k,\nu)$-special presentations with $\nu>1$ (that are not unions of $(m,k)$-special presentations), from which explicit examples of polygonal presentations for $\nu,k$ with $\nu>1$ can be readily obtained.

A cyclic presentation is a group presentation with an equal number of generators and relators that admits a particular cyclic symmetry and the corresponding group is a cyclically presented group (\cite{Johnson97}). A $(3,3,1)$-special cyclic presentation, whose star graph is the Heawood graph was found in~\cite{EdjvetHowie} and it was shown in~\cite[Theorem~3.5]{MohamedWilliams} and~\cite[Theorem~10]{HowieWilliams} that (up to relabelling of generators) this is the only $(3,3,1)$-special cyclic presentation. In this article we show that if a non-redundant cyclic presentation is $(m,k,\nu)$-special then $m=2$ or $3$; we classify the non-redundant $(3,k,\nu)$-special presentations in terms of perfect difference sets (in particular these only arise if the relators are positive or negative words), and we classify the non-redundant $(2,k,\nu)$-special presentations. Except in one unresolved case, we determine which of the groups defined by these presentations are SQ-universal.

We now describe the structure of this article.
In Section~\ref{sec:prelims} we give definitions, terminology and background on group presentations, the star graph, special presentations, and polygonal presentations.
In Section~\ref{sec:stargraphcyclicpresentation} we give examples of $(m,k,\nu)$-special cyclic presentations, prove a theorem concerning the structure of the star graph of a cyclic presentation, and obtain corollaries that are needed for later results; in particular, we show that if $P_n(w)$ is non-redundant and $(m,k,\nu)$-special, where $w$ is positive and $\nu>1$ then the corresponding group $G_n(w)$ is large.
In Section~\ref{sec:girth} we show that if $P_n(w)$ is a non-redundant cyclic presentation where $w$ has length at least~3 then the girth of its star graph is at most $8$ (with girth $>6$ only attainable if $w$ is a non-positive word of length~3) and show that if such a presentation is $(m,k,\nu)$-special then either $m=2$, or $m=3$ and $w$ is a positive or negative word.
In Section~\ref{sec:nonredundant3kalpha} we classify the non-redundant $(3,k,\nu)$-special cyclic presentations in terms of perfect difference sets, and in particular the $(3,3,\nu)$-special cyclic presentations, and we show that a group defined by a non-redundant $(3,k,\nu)$-special cyclic presentation is SQ-universal if and only if $(k,\nu)\neq (3,1)$ (and that there is precisely one such group that is not SQ-universal).
In Section~\ref{sec:nonredundant2kalpha} we classify the non-redundant $(2,k,\nu)$-special cyclic presentations, and in particular the $(2,4,\nu)$-special presentations; we show that, with one possible exception, the corresponding groups are SQ-universal, and we identify which of them define Burger-Mozes groups.

Many of the results from Sections~\ref{sec:stargraphcyclicpresentation},\ref{sec:nonredundant3kalpha},\ref{sec:nonredundant2kalpha} concerning cyclic presentations $P_n(w)$ will be expressed in terms of multisets of differences of subscripts in length 2 cyclic subwords of $w$. For convenience we define these multisets here:
\begin{equation}\label{eq:ABQ}
\left.
\begin{aligned}
\mathcal{A}&=\lbrace a~|~ x_{i}x_{i+a}^{-1},\ 0\leq a<n  ~~\text{is a cyclic subword of } w\rbrace,\\
\mathcal{B}&=\lbrace b~|~ x_{i}^{-1}x_{i+b},\ 0\leq b<n ~~\text{is a cyclic subword of } w\rbrace,\\
\mathcal{Q}&=\lbrace q~|~ x_{i}x_{i+q} ~\text{or}~ x_{i+q}^{-1}x_{i}^{-1},\ 0\leq q<n ~~\text{is a cyclic subword of } w\rbrace,\\
\mathcal{Q}^+&=\{ q~|\ x_ix_{i+q},\ 0\leq q<n~\mathrm{is~a~cyclic~subword~of}~w\},\\
\mathcal{Q}^-&=\{ q~|\ x_{i+q}^{-1}x_i^{-1},\ 0\leq q<n~\mathrm{is~a~cyclic~subword~of}~w\},
\end{aligned}
\right\}
\end{equation}
where the entries are taken mod~$n$.

\section{Preliminaries}\label{sec:prelims}

\subsection{Presentations of groups}\label{sec:wordsinpresentations}
A word $w$ in generators $x_0,\ldots, x_{n-1}$ is said to be \em positive \em (respectively \em negative\em) if all of the exponents of generators are positive (respectively negative). We shall say that $w$ is \em alternating \em if it has no cyclic subword of the form $(x_ix_j)^{\pm 1}$. A word $w=w(x_0,\ldots ,x_{n-1})$ is \em freely reduced \em if does not contain a subword of the form $x_ix_i^{-1}$ or $x_i^{-1}x_i$; it is \em cyclically reduced \em if all cyclic permutations of it are freely reduced. If $w$ is a cyclically reduced, non-empty, word in a free group $F(X)$ then the unique word $v\in F(X)$ such that $w=v^t$ with $t$ maximal is called the \em root \em of $w$. We shall write $l(w)$ to denote the length of $w$ in $F(X)$.

Following \cite{BogleyShift} given a group presentation $P=\pres{X}{R}$, an element $r\in R$ is said to be \em freely redundant \em if it is freely trivial or if there exists $s\in R$ such that $r$ and $s$ are distinct elements of the free group with basis $X$ and either  $r$ is freely conjugate to $s$ or to $s^{-1}$. A presentation is said to be \em redundant \em if it contains a freely redundant relator. The \em deficiency \em of the presentation $P$ is defined as $\mathrm{def}(P)=|X|-|R|$ and the \em deficiency \em of a group $G$, $\mathrm{def}(G)$, is defined to be the maximum of the deficiencies of all finite presentations defining $G$. A group $G$ is \em SQ-universal \em if every countable group embeds in a quotient of $G$ and it is \em large \em if it has a finite index subgroup that has a non-abelian free homomorphic image; every large group is SQ-universal \cite{PrideLargeness}.

\subsection{Cyclic presentations and cyclically presented groups}\label{sec:cyclicpresentations}

For a positive integer $n$, let $F_n$ be the free group with basis $X=\lbrace x_0, \ldots, x_{n-1}\rbrace$ and let $\theta: F_n\rightarrow F_n$ be the \em shift automorphism \em given by $\theta(x_i)=x_{i+1}$ with subscripts taken modulo $n$. If $w$ is a cyclically reduced word that represents an element in $F_n$ then the presentation
$$P_n(w)=\pres{x_0, \ldots, x_{n-1}}{w,\theta(w), \ldots, \theta^{n-1}(w)}$$
is called a \em cyclic presentation \em and the group $G_n(w)$ it defines is a \em cyclically presented group\em. Without loss of generality we may assume that the generator $x_0$ is a letter of $w$, and we make this assumption throughout this article. Then $P_n(w)$ is said to be \em irreducible \em if the greatest common divisor of $n$ and the subscripts of the generators that appear in $w$ is equal to 1~\cite{EdjvetIrreducibleCyclicPresentations}.

The shift automorphism $\theta$ of a cyclically presented group $G_n(w)$ has exponent $n$ and the resulting $\Z_n$-action on $G_n(w)$ determines the \em shift extension \em $E_n(w) = G_n(w) \rtimes_\theta \Z_n$, which admits a presentation of the form
\(E_n(W)=\pres{x,t}{t^n,W(x,t)} \)
where $W=W(x, t)$ is obtained by rewriting $w$ in terms of the substitutions $x_i=t^ixt^{-i}$, $0\leq i<n$ (see, for example, \cite[Theorem~4]{JWW}). Thus there is a retraction $\nu^0: E_n(W) \rightarrow \Z_n$ given by $\nu^0(t)=t$, $\nu^0(x)=t^0=1$ with kernel $G_n(w)$. Moreover, as shown in~\cite[Section~2]{BogleyShift}, there may be further retractions $\nu^f$ for certain values of $f$ ($0\leq f<n$). Specifically, by~\cite[Theorem~2.3]{BogleyShift} the kernel of a retraction $\nu^f:E_n(W) \rightarrow \Z_n$ given by $\nu^f(t)=t$, $\nu^f(x)=t^f$ is cyclically presented, generated by the elements $y_i=t^ixt^{-(i+f)}$ ($0\leq i<n$).

\subsection{Star graph}\label{sec:stargraph}

Let $P = \pres{X}{R}$ be a group presentation such that every relator $r\in R$ is a cyclically reduced word in the generators. Let $\tilde{R}$ denote the symmetrized closure of $R$, that is, the set of all cyclic permutations of elements in $R\cup R^{-1}$. The \em star graph \em of $P$ is the undirected vertex-labelled graph $\Gamma$ where the vertex set is in one-one correspondence with $X\cup X^{-1}$ and vertices are labelled by the corresponding element of $X\cup X^{-1}$ and where there is an edge joining vertices labelled $x$ and $y$ for each distinct word $xy^{-1}u$ in $\tilde{R}$~\cite[page~61]{LyndonSchupp}. Such words occur in pairs, that is $xy^{-1}u\in \tilde{R}$ implies that $yx^{-1}u^{-1}\in \tilde{R}$. These pairs are called \em inverse pairs \em and the two edges corresponding to them are identified in $\Gamma$. It follows that replacing any relator of a presentation by its root, or removing a redundant relator from a presentation, leaves the star graph unchanged. We refer to vertices in $X$ as \em positive \em vertices and vertices in $X^{-1}$ as \em negative \em vertices.

We now set out our graph theoretic terminology. Given a graph $\Gamma$ we write $V(\Gamma)$ to denote its vertex set. If $\Gamma$ is bipartite with vertex partition $V(\Gamma)=V_1\cup V_2$ where each edge connects a vertex in $V_1$ to a vertex in $V_2$ then $V_1,V_2$ are called the \em parts \em of $V(\Gamma)$. The \em neighbours \em of a vertex $v$, denoted by $N_\Gamma(v)$ is the set of all vertices that are adjacent to $v$ in $\Gamma$. A graph $\Gamma$ is \em $q$-regular \em if $|N_\Gamma(v)|=q$ for all $v\in V(\Gamma)$ and it is \em regular \em if it is $q$-regular for some~$q$. We allow graphs to have loops and to have more than one edge joining a pair of vertices.

A \em path \em of \em length \em $l$ in $\Gamma$ is a sequence of vertices $(u=u_0,u_1,\ldots ,u_l=v)$ with edges $u_i-u_{i+1}$ for each $0\leq i<l$; it is a \em closed path \em if $u=v$. The path is \em reduced \em if the edge $u_{i+1}-u_{i+2}$ is not equal to the edge $u_{i+1}-u_{i}$ ($0\leq i<l-1$). The \em distance \em  $d_\Gamma (u,v)$ between vertices $u,v$ of $\Gamma$ is $l\geq 0$ if there is a path of length $l$ from $u$ to $v$, but no shorter path, and $d_\Gamma(u,v)=\infty$ if there is no path from $u$ to $v$. The \em girth\em, $\mathrm{girth}(\Gamma)$ of a graph $\Gamma$ is the length of a reduced closed path of minimal length, if $\Gamma$ contains a reduced closed path, and $\mathrm{girth}(\Gamma)=\infty$ otherwise. The \em diameter\em, $\mathrm{diam}(\Gamma)$ of a graph $\Gamma$ is the greatest distance  between any pair of vertices of the graph (which may be infinite). If $\Gamma$ is a graph with finite girth then $\mathrm{girth}(\Gamma)\leq 2\mathrm{diam}(\Gamma)+1$.

\subsection{Special presentations}\label{sec:specialpresentation}

The concept of $(m,k)$-special presentations was introduced in~\cite{EdjvetVdovina}, generalizing the concept of special presentations, introduced in~\cite{Howie89} (which corresponds to the case $m=k=3$). We extend this to define $(m,k,\nu)$-special presentations, which reduces to $(m,k)$-special in the case $\nu=1$.

\begin{defn}\label{def:mkalphaspecialpres}
Let $m\geq 2, k\geq 3,\nu\geq 1$. A  finite group presentation $P = \pres{X}{R}$ is said to be \em $(m, k,\nu)$-special \em if the following conditions hold:
\begin{itemize}
  \item[(a)] the star graph $\Gamma$ of $P$ has $\nu$ isomorphic components, each of which is a connected, bipartite graph of diameter $m$ and girth $2m$ in which each vertex has degree at least $3$;
  \item[(b)] each relator $r\in R$ has length $k$;
  \item[(c)] if $m=2$ then $k\geq 4$.
\end{itemize}
\end{defn}

Note that the presentations considered in~\cite{EdjvetVdovina} are non-redundant; however, our definition of $(m,k,\nu)$-special (as with the definition of special in~\cite{Howie89}) does not require the presentation to be non-redundant. Note also that if a presentation is $(m,k,\nu)$-special then it has at least 3 generators and the relators are cyclically reduced (for otherwise the star graph contains loops, so is not bipartite). We reiterate and expand on some remarks from~\cite{EdjvetVdovina} concerning $(m,k,\nu)$-special presentations and their star graphs.

\begin{remark}\label{rem:mkalphaspecial}
Let $P$ be an $(m,k,\nu)$-special presentation with star graph $\Gamma$.
\begin{itemize}
  \item[1.] By~\cite[Lemma~1.3.6]{Maldeghem} condition (a) is equivalent to $\Gamma$ having $\nu$ isomorphic components, each of which is the incidence graph of a generalized $m$-gon and thus, by~\cite{FeitHigman}, $m\in \{2,3,4,6,8\}$. The incidence graph $\Lambda$ of a generalized 2-gon is a complete bipartite graph (\cite[page~11]{Maldeghem}, \cite[Section~A.1]{Kantor}) so if, in addition, $\Lambda$ is $k$-regular, then it is the complete bipartite graph $K_{k,k}$. The incidence graph of a generalized 3-gon (or projective plane) of order $q-1$ (for some $q\geq 3$) is bipartite, $q$-regular, each part has $q^2-q+1$ vertices and any two vertices from the same part have exactly one common neighbour (see, for example, \cite[page~373]{Biggs}, \cite[Section~A.1]{Kantor}). Moreover, if $\Lambda$ is a bipartite graph of girth greater than $2$, in which every vertex has degree at least 3 and every pair of vertices from the same part have exactly one common neighbour then $\Lambda$ has girth 6 and diameter~3. As described in~\cite{SingerFiniteProjGeom}, given a perfect difference set of order $q$, it is possible to construct a projective plane of order $q-1$. (A set of $k$ integers $d_1,\ldots ,d_k$ is called a \em perfect difference set \em (of order $k$) if among the $k(k-1)$ differences $d_i-d_j$ each of the residues $1,2,\ldots , (k^2-k)$~mod~$k^2-k+1$ occurs exactly once. For instance $\{1,2,4\}$ and $\{0,1,3,9\}$ are perfect difference sets; further examples can be found in~\cite[Section~7]{Stevenson}.)

  \item[2.]
      Since each vertex of $\Gamma$ has degree $>1$ each generator and its inverse is a piece and since $\Gamma$ has girth $>2$ there are no pieces of length 2 (see, for example, \cite[Section~5]{PrideStarComplexes}) and therefore $P$ satisfies the small cancellation condition $C(3)$~\cite[Chapter~V]{LyndonSchupp}.
\item[3.] By~\cite{HillPrideVella}, if a presentation satisfies $C(3)$ then the $T(q)$ condition is equivalent to the statement that its star graph does not contain any reduced closed path of length $l$ where $3\leq l<q$. Therefore the $(m,k,\nu)$-special presentation $P$ satisfies the small cancellation condition $C(k)-T(2m)$.
  \end{itemize}
\end{remark}

As in~\cite[Proof of Theorem~2]{EdjvetVdovina}, every group defined by an $(m,k,\nu)$-special presentation with $1/m+2/k<1$ is non-elementary hyperbolic (by~\cite[Corollary~4.1]{GerstenShortI},\cite{Collins73},\cite{EdjvetHowie}), and hence SQ-universal (by~\cite[Theorem~1]{OlshanskiiSQ},\cite[Th\'eor\`eme~3.5]{Delzant}). We shall refer to the cases $1/m+2/k=1$ (that is, the cases $(m,k)=(3,3)$ and $(m,k)=(2,4)$) as the \em Euclidean \em cases. It was shown in~\cite[Theorem~2]{EdjvetVdovina} that groups defined by non-redundant $(3,3,1)$-special presentations are just-infinite (and hence not SQ-universal) and that, by~\cite{Collins73}, groups defined by non-redundant $(2,4,1)$-special presentations contain a non-abelian free subgroup and that there are examples (from~\cite{RattaggiThesis,Rattaggi07}) of $(2,4,1)$-special presentations defining both SQ-universal and non-SQ-universal groups.

Since the direct product of two free groups $F_n\times F_n$ ($n\geq 2$) contains a finitely generated subgroup that has undecidable membership problem~\cite{Mihailova} (or see~\cite[IV, Theorem~4.3]{LyndonSchupp}), groups that contain $F_2\times F_2$ as a subgroup are of interest as they fail to satisfy certain properties, such as subgroup separability~\cite{Malcev} and coherence~\cite{Grunewald} and, moreover, can be considered to ``strongly fail'' to be hyperbolic~\cite{BigdelyWise}. Hyperbolic groups and groups defined by $C(3)-T(6)$ presentations do not contain $F_2\times F_2$ (\cite[Theorem~9.3.1]{Bigdely}) so if a group defined by a $(m,k,\nu)$-special presentation contains $F_2\times F_2$ then $(m,k)=(2,4)$. Groups defined by $(2,4,\nu)$-special presentations, however, can contain such a subgroup.

If $v$ is the root of a cyclically reduced word $w$ and $w=v^p$ then the star graph $\Gamma$ of $P_n(w)$ is equal to the star graph of $P_n(v)$, and if $v$ has length 2 then the vertices of $\Gamma$ have degree at most 2, so $P_n(w)$ is $(m,pk,\nu)$-special if and only if $P_n(v)$ is $(m,k,\nu)$-special. Thus, in classifying $(m,k,\nu)$-special cyclic presentations $P_n(w)$ we can assume that $w$ is not a proper power.

\subsection{Polygonal presentations}\label{sec:polygonalpresentations}

In~\cite{Vdovina02,EdjvetVdovina} \em $\lambda$-polygonal presentations \em were defined, and a process for obtaining a corresponding 2-complex $K$ from such a presentation was given; we refer the reader to~\cite{EdjvetVdovina} for the precise definition. In that article, a \em polyhedron \em is defined to be a closed, connected 2-complex $K$ obtained by identifying edges of a given set of $k$-gons ($k\geq 3$); if the 2-complex $K$ obtained from a $\lambda$-polygonal presentation $\mathcal{K}$ is a polyhedron, we say that $K$ \em corresponds \em to $\mathcal{K}$. Lemma~1 of~\cite{EdjvetVdovina} constructs a $\lambda$-polygonal presentation, and a polyhedron that corresponds to it, from any non-redundant $(m,k)$-special presentation. This can be readily extended to deal with $(m,k,\nu)$-special presentations, as in Lemma~\ref{lem:polygonalpresentation}, below. (The hypothesis that the presentation does not decompose as the disjoint union of two non-trivial sub-presentations ensures that the set of tuples $\mathcal{K}$ also does not properly decompose in such a manner, and so the resulting $2$-complex $K$ is connected.)

\begin{lemma}[{Compare~\cite[Lemma~1]{EdjvetVdovina}}]\label{lem:polygonalpresentation}
Let $P=\pres{X}{R}$ be a $(m,k,\nu)$-special presentation that does not decompose as the disjoint union of two non-trivial sub-presentations, and suppose that $\Gamma_0,\ldots ,\Gamma_{\nu-1}$ are the components of the star graph $\Gamma$ of $P$. Then there is a $\lambda$-polygonal presentation $\mathcal{K}$ over $\Gamma_0,\ldots ,\Gamma_{\nu-1}$ with corresponding polyhedron $K$ having $\nu$ vertices $v_0,\ldots ,v_{\nu-1}$ such that the link of $v_i$ is $\Gamma_i$ ($0\leq i< \nu$).
\end{lemma}

Let $P,K$ be as in Lemma~\ref{lem:polygonalpresentation}, let $\tilde{K}$ be the universal cover of $K$, and let $G$ be the group defined by~$P$. Then, as explained in~\cite{Vdovina05,EdjvetVdovina}, the following hold. The group $G$ is the fundamental group of $K$, so $G$ acts cocompactly on $\tilde{K}$. When equipped with the metric introduced in \cite[page~165]{BallmannBrin94}, $\tilde{K}$ is a complete metric space of non-positive curvature in the sense of~\cite{AlexandrovBusemann}, and it is a hyperbolic building if $1/m+2/k<1$ and a Euclidean building if $(m,k)=(2,4)$ or $(3,3)$ \cite{GaboriauPaulin,BallmannBrinOrbihedra,Barre}. Hence $G$ is non-hyperbolic if and only if $(m,k)=(3,3)$ or $(2,4)$ by~\cite[page~119]{GaboriauPaulin},\cite{Bridson}.

\section{Star graphs of cyclic presentations}\label{sec:stargraphcyclicpresentation}

The following example shows that $(m,k,\nu)$-special presentations are prevalent within the class of cyclic presentations.

\begin{example}\label{ex:specialcyclicpresentations}
 \begin{itemize}
   \item[(a)] $P_7(x_0x_1x_3)$ is $(3,3,1)$-special (\cite[Example~3.3]{EdjvetHowie}, \cite[Example~6.3]{Howie89}); its star graph is the Heawood graph (the incidence graph of a projective plane of order~2). The corresponding triangle presentation appears in~\cite[Section~4]{CMSZ1}, \cite[Section~4]{CMSZ2} and the group $G_7(x_0x_1x_3)$ also appears in~\cite[Section~3]{BBPV}, and the proof of~\cite[Theorem~2]{EdjvetVdovina}.
   In particular, it is known that $G_7(x_0x_1x_3)$ is just-infinite (so is not SQ-universal) and non-hyperbolic -- see~\cite[Example~3.8]{MohamedWilliams} for a discussion.

   \item[(b)] $P_{21}(x_0x_1x_5)$ is $(3,3,3)$-special; its star graph has 3 components, each of which is the Heawood graph -- see Figure~\ref{fig:threeHeawoods}. The group $G_{21}(x_0x_1x_5)$ is large (by~\cite[Lemma~2.3]{EdjvetWilliams}) and is not hyperbolic see~\cite[Corollary~2.10]{ChinyereWilliamsT6}. The free product of three copies of $G_7(x_0x_1x_3)$ is the group $G_{21}(x_0x_3x_9)$ and the shift extension of $G_{21}(x_0x_3x_9)$ is isomorphic to the shift extension of $G_{21}(x_0x_1x_5)$, so the structures of $G_7(x_0x_1x_3)$ and $G_{21}(x_0x_1x_5)$ are related through this shift-extension. For example, in~\cite[Corollary~2.10]{ChinyereWilliamsT6} the fact that $G_7(x_0x_1x_3)$ is non-hyperbolic is used to prove that $G_{21}(x_0x_1x_5)$ is non-hyperbolic.

   \item[(c)] $P_{13}(x_0^2x_1x_4)$ is $(3,4,1)$-special; its star graph is the  $(4,6)$-cage~\cite{Wong} (the incidence graph of a projective plane of order~3) -- see Figure~\ref{fig:(3,4,1)specialstargraph}.

   \item[(d)] $P_4(x_0x_1x_0^{-1}x_1^{-1})$ (defining $F_2\times F_2$) is $(2,4,1)$-special; its star graph is the complete bipartite graph $K_{4,4}$.

   \item[(e)] The presentation $P_7(x_0^2x_1 x_4^2x_5 x_1^2x_2 x_5^2x_6 x_2^2x_3 x_6^2x_0 x_3^2x_4)$ is (redundant and) $(3,21,1)$-special; its star graph is the Heawood graph.
 \end{itemize}
\end{example}

\begin{figure}
\begin{center}
\begin{tabular}{lll}
\begin{tikzpicture}%

\SetVertexNoLabel
\grHeawood[RA=2]
\AssignVertexLabel[size  = \footnotesize]{a}{$x_0$,$x_1^{-1}$,$x_{18}$,$x_{13}^{-1}$,$x_{12}$,$x_{16}^{-1}$,$x_{15}$,$x_{19}^{-1}$,$x_3$,$x_7^{-1}$,$x_6$,$x_{10}^{-1}$,$x_9$,$x_4^{-1}$}
\end{tikzpicture}
&
\begin{tikzpicture}%
\SetVertexNoLabel
\grHeawood[RA=2]
\AssignVertexLabel[size  = \footnotesize]{a}{$x_1$,$x_2^{-1}$,$x_{19}$,$x_{14}^{-1}$,$x_{13}$,$x_{17}^{-1}$,$x_{16}$,$x_{20}^{-1}$,$x_4$,$x_8^{-1}$,$x_7$,$x_{11}^{-1}$,$x_{10}$,$x_5^{-1}$}
\end{tikzpicture}
&
\begin{tikzpicture}%
\SetVertexNoLabel
\grHeawood[RA=2]
\AssignVertexLabel[size  = \footnotesize]{a}{$x_2$,$x_3^{-1}$,$x_{20}$,$x_{15}^{-1}$,$x_{14}$,$x_{18}^{-1}$,$x_{17}$,$x_{0}^{-1}$,$x_5$,$x_9^{-1}$,$x_8$,$x_{12}^{-1}$,$x_{11}$,$x_6^{-1}$}
\end{tikzpicture}
\end{tabular}

\end{center}
  \caption{The star graph of $P_{21}(x_0x_1x_5)$ (the disjoint union of three Heawood graphs).\label{fig:threeHeawoods}}
\end{figure}
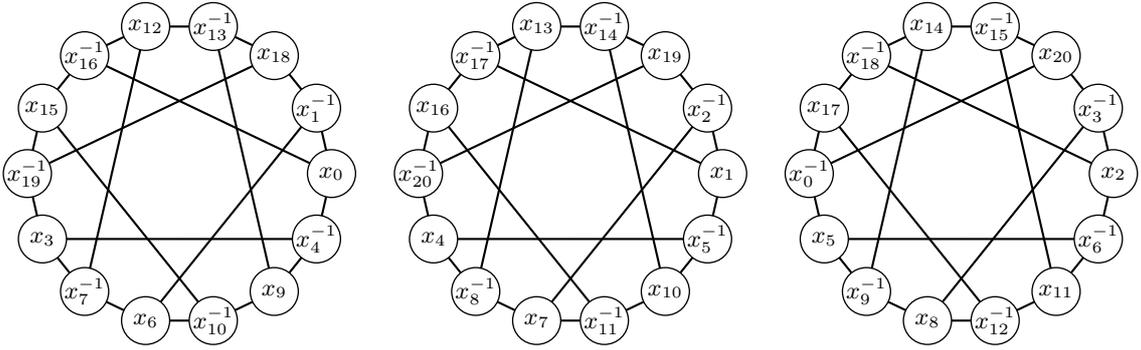

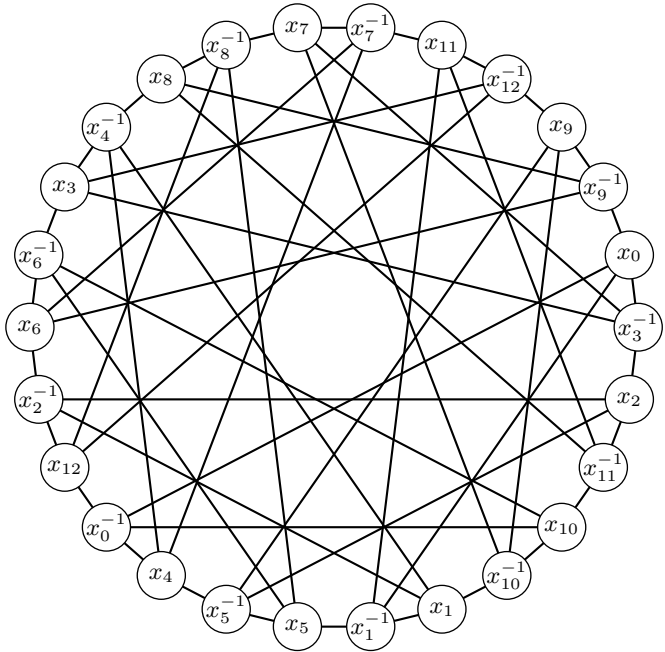
\begin{figure}
\begin{center}
\begin{tikzpicture}%
\SetVertexNoLabel
\grLCF[RA=4]{7,-11}{13}
\AssignVertexLabel[size  = \footnotesize]{a}{$x_3^{-1}$,$x_0$,$x_9^{-1}$,$x_9$,$x_{12}^{-1}$,$x_{11}$,$x_7^{-1}$,$x_7$,$x_8^{-1}$,$x_8$,$x_4^{-1}$,$x_3$,$x_6^{-1}$,$x_6$,$x_2^{-1}$,$x_{12}$,$x_0^{-1}$,$x_4$,$x_5^{-1}$,$x_5$,$x_1^{-1}$,$x_1$,$x_{10}^{-1}$,$x_{10}$,$x_{11}^{-1}$,$x_2$}
\end{tikzpicture}
\end{center}
  \caption{The star graph of $P_{13}(x_0^2x_1x_4)$ (the $(4,6)$-cage).\label{fig:(3,4,1)specialstargraph}}
\end{figure}

We use Example~\ref{ex:specialcyclicpresentations}(b) to give an example of a $\lambda$-polygonal presentation over a disconnected graph:

\begin{example}[A polygonal presentation over the union of three copies of the Heawood graph]\label{ex:polygonalpres3Heawoods}
Let $\Gamma$ be the graph in Figure~\ref{fig:threeHeawoods} and let $U_1=\{x_i\ |\ 0\leq i<21\}$, $U_2=\{x_i^{-1}\ |\ 0\leq i<21\}$, and $\lambda:U_1\rightarrow U_2$ be the bijection given by $\lambda(x_i)=x_i^{-1}$. Then the set
\[ \mathcal{K}=\{ (x_i,x_{i+1},x_{i+5}), (x_{i+1},x_{i+5},x_i), (x_{i+5}, x_i,x_{i+1}) \ |\ 0\leq i<21 \}\]
(subscripts mod~$21$) is a $\lambda$-polygonal presentation over $\Gamma$.
\end{example}

Let $\delta$ denote the greatest common divisor of $n$ and the subscripts of the generators that appear in $w$ (recall that by our standing assumption $x_0$ is involved in $w$). Then the cyclic presentation $P=P_n(w(x_0,x_\delta,\dots ,x_{(n/\delta-1)\delta}))$ decomposes as the disjoint union of $\delta$ cyclic presentations $R=P_{n/\delta}(w(x_0,x_1,\dots ,x_{n/\delta-1}))$. Hence the star graph of $P$ decomposes as the disjoint union of $\delta$ copies of the star graph of $R$, and so $P$ is $(m,k,\delta\nu)$-special if and only if $R$ is $(m,k,\nu)$-special. Thus, it is often convenient to assume that $P_n(w)$ is irreducible (i.e.\,that $\delta=1$).

The following theorem (compare~\cite[Lemma~2.3]{HowieWilliamsPlanarity}) describes the components of the star graph of a non-redundant cyclic presentation $P_n(w)$. For an integer $n\geq 2$ and subset $A\subseteq \{0,1,\ldots ,n-1\}$, the \em circulant graph \em $\mathrm{circ}_n(A)$ is the graph with vertices $v_0,\ldots ,v_{n-1}$ and edges $v_i-v_{i+a}$ for all $0\leq i<n$, $a\in A$ (subscripts mod~$n$).

\begin{theorem}\label{thm:components}
Let $\Gamma$ be the star graph of a non-redundant cyclic presentation $P_n(w)$ where $w$ is cyclically reduced and is not a proper power, let $\mathcal{A},\mathcal{B},\mathcal{Q}$ be the multisets defined at~(\ref{eq:ABQ}), let $d_\mathcal{A}=\mathrm{gcd}(n, a\ (a\in\mathcal{A}))$, $d_\mathcal{B}=\mathrm{gcd}(n, b\ (b\in\mathcal{B}))$, and if $w$ is non-alternating let $q_0\in\mathcal{Q}$ and set $d=\mathrm{gcd}(n,a\ (a\in \mathcal{A}),b\ (b\in \mathcal{B}), q-q_0\ (q\in \mathcal{Q}))$. Then $\Gamma$ is $l(w)$-regular and has vertices $x_i,x_i^{-1}$ ($0\leq i<n$) and edges $x_i-x_{i+a}, x_i^{-1}-x_{i+b}^{-1}, x_i-x_{i+q}^{-1}$ for all $a\in\mathcal{A}, b\in\mathcal{B}, q\in\mathcal{Q}$, $0\leq i<n$.
\begin{itemize}
  \item[(a)] If $w$ is non-alternating then $\Gamma$ has $d$ isomorphic connected components $\Gamma_0,\ldots ,\Gamma_{d-1}$ where for $0\leq j<d$ the graph $\Gamma_j$ is the induced labelled subgraph of $\Gamma$ with vertex set $V(\Gamma_j) = V(\Gamma_j^+) \cup V(\Gamma_j^-)$ where $\Gamma_j^+$ and $\Gamma_j^-$ are the induced labelled subgraphs of $\Gamma$ with vertex sets
\begin{alignat*}{1}
V(\Gamma_j^+)&= \{ x_{j+td}\ |\ 0\leq t< n/d\},\\
V(\Gamma_j^-)&= \{x_{j+q_0+td}^{-1}\ |\ 0\leq t< n/d\}
\end{alignat*}
(subscripts mod~$n$). In particular $|V(\Gamma_j^+)|=|V(\Gamma_j^-)|= n/d$ for all $0\leq j<d$ and the subscripts of the positive (respectively negative) vertices in any component are congruent mod~$d$.

\item[(b)] If $w$ is alternating then $\Gamma$ has $d_\mathcal{A} +d_\mathcal{B}$ connected components $\Gamma_0^+,\ldots ,\Gamma_{d_\mathcal{A}-1}^+$ and $\Gamma_0^-,\ldots ,\Gamma_{d_\mathcal{B}-1}^-$ which are, respectively, the induced labelled subgraphs of $\Gamma$ with vertex sets
\begin{alignat*}{1}
V(\Gamma_j^+)&= \{ x_{j+td_\mathcal{A}}\ |\ 0\leq t< n/d_\mathcal{A}\},\\
V(\Gamma_j^-)&= \{x_{j+td_\mathcal{B}}^{-1}\ |\ 0\leq t< n/d_\mathcal{B}\}
\end{alignat*}
(subscripts mod~$n$). Moreover
each graph $\Gamma_j^+$ is isomorphic to the circulant graph \linebreak $\mathrm{circ}_{n/d_\mathcal{A}}(\{a/d_\mathcal{A}\ (a\in\mathcal{A})\})$ and
each graph $\Gamma_j^-$ is isomorphic to the circulant graph \linebreak $\mathrm{circ}_{n/d_\mathcal{B}}(\{b/d_\mathcal{B}\ (b\in\mathcal{B})\})$.
\end{itemize}
\end{theorem}

\begin{proof}
Observe first that each positive vertex has the same degree and, since in the star graph of any finite presentation vertices corresponding to a generator and its inverse have the same degree (see~\cite[Section~2.3.3]{HillPrideVella}), the graph $\Gamma$ is regular. Moreover, the number of edges of the star graph of a non-redundant presentation that has no proper power relators, and where the relators are cyclically reduced, is equal to the sum of the lengths of the relators, so the number of edges of $\Gamma$ is equal to $nl(w)$, and hence $\Gamma$ is $l(w)$-regular.

Let $\Gamma^+$, $\Gamma^-$ denote the induced subgraphs of $\Gamma$ with positive and negative vertices, respectively. Then $\Gamma^+$ has vertices $x_0,\ldots,x_{n-1}$ and edges $x_i-x_{i+a}$ for each $a\in\mathcal{A}$, $0\leq i<n$, so is the circulant graph $\mathrm{circ}_n(\mathcal{A})$. Therefore its connected components are the graphs $\Gamma_0^+,\ldots ,\Gamma_{d_\mathcal{A}-1}^+$, each of which is isomorphic to the circulant graph $\mathrm{circ}_{n/d_\mathcal{A}}(\{a/d_\mathcal{A}\ (a\in\mathcal{A})\})$ (see, for example, \cite{BoeschTindell},\cite[page~154]{Heuberger}). Similarly $\Gamma^-$ is the circulant graph  $\mathrm{circ}_n(\mathcal{B})$, and its connected components are the graphs $\Gamma_0^-,\ldots ,\Gamma_{d_\mathcal{B}-1}^-$, each of which is isomorphic to the circulant graph $\mathrm{circ}_{n/d_\mathcal{B}}(\{b/d_\mathcal{B}\ (b\in\mathcal{B})\})$. Thus part~(b) is proved so consider part~(a). Fix a $j$, $0\leq j<d$ and consider the graph $\Gamma_j$. Identifying the endpoints of the edges $x_{j+td}-x_{j+q_0+td}^{-1}$ ($0\leq t<n/d$) of $\Gamma_j$ leaves the circulant graph $\Lambda_j$ with vertices $x_j,x_{j+d},\ldots , x_{j+n-d}$ and edges $x_{j+td}-x_{j+td}$, $x_{j+td}-x_{j+td+a}$, $x_{j+td}-x_{j+td+b}$, $x_{j+td}-x_{j+ ((q-q_0)/d+t)d }$   for all $a\in \mathcal{A}, b\in \mathcal{B}, q\in \mathcal{Q}$,  $0\leq t<n/d$ (subscripts mod~$n$). Setting $u_t=x_{j+td}$ for each $0\leq t<n/d$ the graph $\Lambda_j$ has vertices $u_0,\ldots , u_{n/d-1}$ and (multi-)edges $u_t-u_{t}, u_t-u_{t+a/d}, u_t-u_{t+b/d}, u_t-u_{t+(q-q_0)/d}$ (subscripts mod~$n/d$), which is connected since $\mathrm{gcd}(n/d,a/d\ (a\in \mathcal{A}),b/d\ (b\in \mathcal{B}), (q-q_0)/d\ (q\in \mathcal{Q}))=1$. Therefore $\Lambda_j$, and hence $\Gamma_j$, is connected, as required.
\end{proof}

As an immediate corollary we have:

\begin{corollary}\label{cor:mkalphacharacterisation}
Let $P_n(w)$ be a non-redundant cyclic presentation in which $w$ is not a proper power. Then
\begin{itemize}
  \item[(a)] $P_n(w)$ is $(3,k,\nu)$-special if and only if $k^2-k+1=n/\nu$ and each component of its star graph is the incidence graph of a projective plane of order $k-1$;
 \item[(b)] $P_n(w)$ is $(2,k,\nu)$-special if and only if $k=n/\nu$ and each component of its star graph is the complete bipartite graph $K_{k,k}$.
  \end{itemize}
\end{corollary}

We also have the following:

\begin{corollary}\label{cor:oddnnotspecial}
Suppose that $P_n(w)$ is a non-redundant $(m,k,\nu)$-special cyclic presentation in which $w$ is non-positive, non-negative, and non-alternating, and is not a proper power. Then $n/\nu$ is even and, in particular, $P_n(w)$ is not $(3,k,\nu)$-special for any $k\geq 3$, $\nu\geq 1$.
\end{corollary}

\begin{proof}
Since $w$ is non-alternating, in the notation of Theorem~\ref{thm:components}, $\nu=d$ so $n/\nu$ is an integer. Let $a\in \mathcal{A}$ (which is non-empty, since $w$ is non-positive and non-negative) and let $r=\mathrm{gcd}(n/\nu,a)$. Then $x_0-x_a-x_{2a}-\cdots -x_{((n/\nu)/r-1)a}-x_0$ is a closed path in the star graph $\Gamma$ of $P_n(w)$ of length $(n/\nu)/r$, which is even, since each component of $\Gamma$ is bipartite. Hence $n/\nu$ is even. If $P_n(w)$ is $(3,k,\nu)$-special for some $k\geq 3$, $\nu\geq 1$ then by Corollary~\ref{cor:mkalphacharacterisation}(a) $n/\nu=k^2-k+1$, which is odd, a contradiction.
\end{proof}

\begin{corollary}\label{cor:disconnectedlarge}
Let $P_n(w)$ be a non-redundant cyclic presentation and let $\Delta$ be the number of components of the star graph of $P_n(w)$ and let $\sigma$ be the exponent sum of $w$. If either
\begin{itemize}
  \item[(a)] $w$ is non-alternating and $\Delta>1$ and $(\Delta,|\sigma|)\neq (2,2)$; or
  \item[(b)] $w$ is alternating and $\Delta>2$
\end{itemize}
then $G_n(w)$ is large. In particular, if $P_n(w)$ is $(m,k,\nu)$-special where $w$ is a positive word and $\nu>1$, then $G_n(w)$ is large.
\end{corollary}

\begin{proof}
The values of $\sigma$ and $\Delta$ do not depend on whether or not $w$ is cyclically reduced, so we may assume that it is. Moreover, if $v$ is the root of $w$ then if $G_n(v)$ is large then so is $G_n(w)$, and the star graph of $P_n(w)$ is equal to that of $P_n(v)$ so we may assume that $w$ is not a proper power.

Suppose first that $w$ is alternating. Then, in the notation of Theorem~\ref{thm:components}, $d_\mathcal{A}>1$ or $d_\mathcal{B}>1$. By adjoining the relators $x_ix_{i+d_\mathcal{A}}^{-1}$ ($0\leq i<n)$ we see the group $G_n(w)$ maps onto the free group of rank $d_\mathcal{A}$, and adjoining the relators $x_ix_{i+d_\mathcal{B}}^{-1}$ ($0\leq i<n)$ it maps onto the free group of rank $d_\mathcal{B}$, and hence $G_n(w)$ is large.

Suppose then that $w$ is not alternating, so that (in the notation of Theorem~\ref{thm:components}) $\Delta=d$. By inverting and cyclically permuting $w$, if necessary, we may assume that the exponent sum $\sigma \geq 0$ and the initial letter of $w$ is $x_0$. Now $a\equiv 0,b\equiv 0, q\equiv q_0$~mod~$d$ for all $a\in \mathcal{A}, b\in \mathcal{B}, q\in \mathcal{Q}$, so if $x_i^{\epsilon_i}ux_j^{\epsilon_j}$ is a cyclic subword of $w$, where $u$ is alternating and $\epsilon_i,\epsilon_j\in\{\pm 1\}$ then $i\equiv j$~mod~$d$ if $\epsilon_i=-\epsilon_j$, $j\equiv i+q_0$~mod~$d$ if $\epsilon_i=\epsilon_j=1$, and  $j\equiv i-q_0$~mod~$d$ if $\epsilon_i=\epsilon_j=-1$.
Therefore, by adjoining the relators $x_ix_{i+d}^{-1}$ ($0\leq i<n)$ the group $G_n(w)$ maps onto  $G_d(w')$, where $w'=x_0x_{q_0}\ldots x_{(\sigma-1)q_0}$, which has shift extension $E=G_d(w')\rtimes \Z_d=\pres{x,t}{t^d,(xt^{q_0})^{\sigma}}\cong \Z_d*\Z_\sigma$. Thus $E$, and hence $G_n(w)$, is large if $\sigma\neq 1$ and $(d,\sigma)\neq (2,2)$.
\end{proof}

Note that Corollary~\ref{cor:disconnectedlarge}(a) cannot be directly extended to include the case $(\Delta,\sigma)=(2,2)$ since, for example, for even $n$ the group $G_n(x_0x_1)\cong \Z$, and, similarly, part~(b) cannot be directly extended to $\Delta=2$ since, for example, $G_n(x_0x_1^{-1})\cong \Z$. In the following example we give an infinite family of $(2,k,1)$-special cyclic presentations that define large groups.

\begin{example}\label{ex:2k1large}
Let $w=x_0x_{0+1}x_{0+1+2}\ldots x_{0+1+2+\dots+n-1}$, where $n>4$ is odd and composite. Then the star graph of $P_n(w)$ is the complete bipartite graph $K_{n,n}$. Let $d$ be a proper divisor of $n$. Adjoining the relators $x_ix_{i+d}^{-1}$ ($0\leq i<n$) to $P_n(w)$ shows that $G_n(w)$ maps onto $G_d(w')$, where $w'=(x_0x_{0+1}x_{0+1+2}\ldots x_{0+1+2+\dots+d-1})^{n/d}$. Since either $d\geq 3$ or $d=2$ and $n/d\geq 3$ the group $G_d(w')$ is large by~\cite{EdjvetBalanced}.
\end{example}

\section{Girth of the star graph of a cyclic presentation}\label{sec:girth}

In this section we prove:

\begin{maintheorem}\label{thm:m=2or3}
Suppose that the cyclic presentation $P_n(w)$ is non-redundant, where $w$ is non-negative, has length at least 3, and is not a proper power. Then,
\begin{itemize}
  \item[(a)] if $P_n(w)$ satisfies $T(q)$ where $q\geq 7$ then $l(w)=3$, $q\leq 8$, and $w$ is non-positive;
  \item[(b)] if $w$ is a non-positive word of length $k=3$ then $P_n(w)$ is not $(m,k,\nu)$-special for any $m\geq 2$, $\nu\geq 1$.
\end{itemize}
Hence, if $P_n(w)$ is $(m,k,\nu)$-special for some $m\geq 2$, $k\geq 3$, $\nu\geq 1$ then either $m=2$ or ($m=3$ and $w$ is positive).
\end{maintheorem}

The `non-redundant' hypothesis is necessary in Theorem~\ref{thm:m=2or3}(a) since, for example, $P=P_6(x_0x_1x_3x_4)$ is redundant and its star graph is a 12-cycle and so $P$ satisfies $T(12)$.

\begin{lemma}\label{lem:girthupperboundtechnical}
Let $\Gamma$ be the star graph of a non-redundant cyclic presentation $P_n(w)$ where $w\in F_n$ is cyclically reduced and is not a proper power. Then
\begin{itemize}
\item[(a)] if $w$ contains non-overlapping cyclic subwords of the form $x_jx_{j+p}^\epsilon, (x_kx_{k+p}^\epsilon)^{\pm 1}$ $0\leq j,k,p<n$, where $\epsilon=\pm 1$, then $\mathrm{girth}(\Gamma)\leq 2$.
  \item[(b)] if $w$ is a positive word of length~3 then $\mathrm{girth}(\Gamma)\leq 6$;
  \item[(c)] if $w$ contains a cyclic subword of the form $x_jx_{j+p}x_{j+p+q}^{-1}x_{j+p+q+r}^{-1}$ (for some $0\leq j,p,q,r<n$) then $\mathrm{girth}(\Gamma)\leq 6$;
  \item[(d)] if $w$ contains a positive cyclic subword of length~4 then $\mathrm{girth}(\Gamma)\leq 6$;
  \item[(e)] if $w$ contains a cyclic subword of the form $x_jx_{j+p}x_{j+p+q}x_{j+p+q+r}^{-1}$ (for some $0\leq j,p,q,r<n$) then $\mathrm{girth}(\Gamma)\leq 6$;
  \item[(f)] if $w$ contains non-overlapping cyclic subwords of the form $x_jx_{j+p}^{-1}$ and $x_kx_{k+q}^{-1}$ for some $0\leq j,k,p,q<n$, then $\mathrm{girth}(\Gamma)\leq 4$.
\end{itemize}
\end{lemma}

\begin{proof}
\begin{itemize}
  \item[(a)] Here the edges of $\Gamma$ includes the distinct edges $x_i-x_{i+p}^{-\epsilon}$, $x_i-x_{i+p}^{-\epsilon}$ ($0\leq i<n$) so $\Gamma$ contains the reduced closed path $x_{0}-x_{p}^{-\epsilon}-x_0$ of length~2.

  \item[(b)] We may assume $w=x_0x_{p}x_{p+q}$ for some $0\leq p,q< n$ where $p,q$ are not both $0$ and $p\not \equiv q$, $p+2q\not \equiv 0$, $2p+q\not \equiv 0$~mod~$n$ (by part (a)). Then $\Gamma$ contains the reduced closed path $x_0- x_p^{-1}- x_{p-q}- x_{-2q}^{-1}- x_{-p-2q}- x_{-p-q}^{-1}- x_0$ of length~6.
  \item[(c)] By part~(a) we may assume $p+r\not \equiv 0$~mod~$n$. Then $\Gamma$ contains the reduced closed path $x_0- x_p^{-1}- x_{p+r}- x_{p+q+r}- x_{p+q}^{-1}- x_{q}- x_0$ of length~6.

  \item[(d)] We may assume $w$ contains the cyclic subword $x_0x_{p}x_{p+q}x_{p+q+r}$ for some $0\leq p,q,r<n$, where $p\not \equiv q$, $q\not \equiv r$, and $p\not \equiv r$~mod~$n$ (by part~(a)). Then $\Gamma$ contains the reduced closed path $x_0- x_p^{-1}- x_{p-q}- x_{p-q+r}^{-1}- x_{-q+r}- x_{r}^{-1}- x_0$ of length~6.

  \item[(e)] By part~(a) we may assume $p\not\equiv q$~mod~$n$.  Then $\Gamma$ contains the reduced closed path $x_0- x_p^{-1}- x_{p-q}- x_{p-q+r}- x_{p+r}^{-1}- x_{r}- x_0$ of length~6.

  \item[(f)] By part~(a) we may assume $p\not \equiv \pm q$~mod~$n$. Then $\Gamma$ contains the reduced closed path $x_0- x_p- x_{p+q}- x_{q}- x_0$ of length~4.
\end{itemize}
\end{proof}

\begin{corollary}\label{cor:girthupperbound}
Let $\Gamma$ be the star graph of a non-redundant cyclic presentation $P_n(w)$ where $w$ is cyclically reduced of length at least 3 and is not a proper power. Then $\mathrm{girth}(\Gamma)\leq 8$ and, moreover, if $\mathrm{girth}(\Gamma)>6$ then $w$ is a non-positive, non-negative word of length 3.
\end{corollary}

\begin{proof}
We refer to parts~(a)--(f) of Lemma~\ref{lem:girthupperboundtechnical}. By replacing $w$ by its inverse, if necessary, we may assume that $w$ is non-negative. If $w$ is alternating then $\mathrm{girth}(\Gamma)\leq 4$ by part~(f). If $l(w)=3$ and $w$ is positive the result follows from part~(b). If $l(w)=3$ and $w$ is non-positive and non-negative, $w=x_0x_px_{p+q}^{-1}$, say, then $x_0-x_p^{-1}-x_{-q}^{-1}-x_{-p-q}-x_{-p}-x_0^{-1}-x_{p+q}^{-1}-x_q-x_0$ is a reduced closed path in $\Gamma$ of length 8. Thus $\mathrm{girth}(\Gamma)\leq 8$.

Suppose $l(w)=4$. If $w$ is positive the result follows from part~(d) so assume that $w$ is non-positive. By part~(f) we may assume that $w$ has exactly one cyclic subword of the form $x_jx_{j+p}^{-1}$ and therefore has a cyclic subword of the form of part~(c) or~(e), and the result follows.

Suppose then that $l(w)\geq 5$. If $w$ is positive then the result follows from part~(d) so assume that $w$ is non-positive. By part~(f) we may assume that $w$ has exactly one cyclic subword of the form $x_jx_{j+p}^{-1}$, but then $w$ or its inverse has a cyclic subword of the form given in part~(e).
\end{proof}

We note that girth 8 can be attained in Corollary~\ref{cor:girthupperbound}; a presentation that demonstrates this is $P_{18}(x_0x_8x_1^{-1})$. We now prove Theorem~\ref{thm:m=2or3}.

\begin{proof}[Proof of Theorem~\ref{thm:m=2or3}]
(a) If $P_n(w)$ satisfies $T(q)$ with $q\geq 7$ then every piece has length 1 so $P_n(w)$ satisfies $C(l(w))$, so it satisfies $C(3)$, and hence  the star graph $\Gamma$ of $P_n(w)$ contains no reduced closed path of length $l$ where $3\leq l<q$. Moreover, since every piece has length 1 there is no reduced closed path of length 2 in $\Gamma$. Then by Corollary~\ref{cor:girthupperbound} $q\leq 8$ and $w$ is a non-positive word of length~3.

(b) Suppose that $w$ is a non-positive word of length $k=3$ and that $P_n(w)$ is $(m,k,\nu)$-special for some $m\geq 2$, $\nu\geq 1$. Then, since $k=3$ we have $m\neq 2$, and by part~(a) $m\leq 4$. By Theorem~\ref{thm:components}, each component of $\Gamma$ has $2n/\nu$ vertices and is 3-regular, and since it is the incidence graph of a generalized $m$-gon it has 14 vertices if $m=3$ and 30 vertices if $m=4$ (see~\cite[Corollary~1.5.5]{Maldeghem}). Therefore $n/\nu=7$ or $15$, a contradiction to Corollary~\ref{cor:oddnnotspecial}.
\end{proof}

We obtain our classifications of the non-redundant $(m,k,\nu)$-special cyclic presentations for $m=3$ and $m=2$ in Sections~\ref{sec:nonredundant3kalpha} and \ref{sec:nonredundant2kalpha}, respectively.

\section{Classification of non-redundant $(3,k,\nu)$-special cyclic presentations}\label{sec:nonredundant3kalpha}

In this section, in Theorem~\ref{thm:3kalphapositiveclassification}, we classify the non-redundant $(3,k,\nu)$-special cyclic presentations in terms of perfect difference sets of order $k$. In Corollary~\ref{cor:33alphaspecial} we consider the Euclidean case and list explicitly the $(3,3,\nu)$-special cyclic presentations. By Theorem~\ref{thm:m=2or3} we may assume that $w$ is a positive word.

\begin{maintheorem}\label{thm:3kalphapositiveclassification}
Suppose that $w$ is a positive word of length $k\geq 3$ that is not a proper power, let $\mathcal{Q}$ be the multiset defined at~(\ref{eq:ABQ})
and suppose that the cyclic presentation $P_n(w)$ is irreducible and non-redundant. Then $P_n(w)$ is $(3,k,\nu)$-special if and only if
\begin{itemize}
  \item[(a)] $n=\nu N$, where $N=k^2-k+1$; and
  \item[(b)] $\mathcal{Q}$ is a perfect difference set; and
  \item[(c)] $q\equiv q'$ mod $\nu$, for all $q,q'\in\mathcal{Q}$; and
  \item[(d)] $\nu|k$.
\end{itemize}
\end{maintheorem}

\begin{proof}
Let $\mathcal{Q}=\{q_1,\ldots,q_{k}\}$. Suppose first that $P_n(w)$ is $(3,k,\nu)$-special and let $\Gamma$ denote the star graph of $P_n(w)$. Condition~(a) holds by Corollary~\ref{cor:mkalphacharacterisation} and (c) holds by Theorem~\ref{thm:components}. By (c) for all $1\leq i\leq k$ we have $\mathrm{gcd}(\nu,q_i)=\mathrm{gcd}(\nu,q_1,\ldots,q_k)$ which divides $\mathrm{gcd}(n,q_1,\ldots ,q_k)=1$, since $P_n(w)$ is irreducible. Thus $\mathrm{gcd}(\nu,q_i)=1$ for all $1\leq i\leq k$. Then, again by (c), $kq_i\equiv \sum_{\iota=1}^k q_\iota \equiv 0$~mod~$\nu$, and since $\mathrm{gcd}(\nu,q_i)=1$ we have $k\equiv 0$~mod~$\nu$ and so~(d) holds.

Let $\Lambda$ be the component of $\Gamma$ that contains $x_0^{-1}$. By Theorem~\ref{thm:components} the subscripts of the negative vertices in $\Lambda$ are all congruent to zero mod~$\nu$. Since there are $N$ negative vertices in this component, they are precisely the negative vertices whose subscripts are the elements of the set $\{0,\nu,\ldots ,(N-1)\nu\}$. On the other hand, a vertex $x_j^{-1}$ is a vertex of $\Lambda$ if and only if $N_\Gamma(x_0^{-1})\cap N_\Gamma(x_j^{-1}) = \{x_u\}$ for some vertex $x_u$, say. Then $j\equiv u+q_t$, $0\equiv u+q_s$~mod~$n$ for some $1\leq s,t\leq k$, so $j\equiv q_t-q_s$~mod~$n$. That is, the subscripts of the negative vertices in $\Lambda$ are precisely the elements of the set $\{ (q_t-q_s)~\mathrm{mod}~n\ |\ 1\leq s,t\leq k\}$. Therefore,
\[ \{j \nu~\mathrm{mod}~N\ |\ 0\leq j <N\} = \{ (q_t-q_s)~\mathrm{mod}~N\ |\ 1\leq s,t\leq k\}.\]

Since $k\equiv 0$~mod~$\nu$ we have $\mathrm{gcd}(N,\nu)=\mathrm{gcd}(k^2-k+1,\nu)=1$, and hence
\[ \{0,1,\dots , N-1\} = \{ (q_t-q_s)~\mathrm{mod}~N\ |\ 1\leq s,t\leq k\}.\]
Then, if $q_s\equiv q_t$~mod~$N$ for some $s\neq t$ the set
\( \{ (q_t-q_s)~\mathrm{mod}~N\ |\ 1\leq s,t\leq k, s\neq t\},\)
which has at most $k^2-k$ elements, is equal to the set
\( \{ (q_t-q_s)~\mathrm{mod}~N\ |\ 1\leq s,t\leq k\},\)
which has $N=k^2-k+1$ elements, a contradiction. Therefore $q_s\equiv q_t$~mod~$N$ if and only if $s=t$ so
\[ \{1,\dots , N-1\} = \{ (q_t-q_s)~\mathrm{mod}~N\ |\ 1\leq s,t\leq k, s\neq t\}\]
and hence $\{q_1,\ldots, q_k\}$ is a perfect difference set and so (b) holds.

For the converse, suppose that (a)--(d) hold and let $\Gamma$ be the star graph of $P_n(w)$. We must show that $\Gamma$ has $\nu$ components, each of which is bipartite, has girth 6 and diameter 3, and each vertex has degree at least 3.

Since $w$ is positive the graph $\Gamma$ is bipartite. Condition (b) implies that $q_1,\ldots ,q_k$ are distinct mod $N$ and hence are distinct mod~$n$, so by Theorem~\ref{thm:components} $\Gamma$ is $k$-regular (and hence each vertex has degree at least 3) and has no reduced closed path of length 2. By Theorem~\ref{thm:components} $\Gamma$ has $d=\mathrm{gcd}(n,q_1,\ldots ,q_k)$ isomorphic components $\Gamma_0,\ldots ,\Gamma_{d-1}$, each with $2n/d$ vertices. If we show that $\Gamma_0$ has girth 6 and diameter 3 then each component has $2(k^2-k+1)$ vertices, so $2n/d=2(k^2-k+1)$ so $n/d=k^2-k+1$. But condition (a) implies that $n/\nu=k^2-k+1$ so $d=\nu$, i.e.\,$\Gamma$ has $\nu$ components.

We now show this. It suffices to show that $|N_\Gamma(u)\cap N_\Gamma(v)|=1$ whenever $u,v$ are distinct vertices belonging to the same part of $\Gamma_0$ (see Remark~\ref{rem:mkalphaspecial}(1)). Let $\epsilon=\pm 1$ and suppose $x_i^{\epsilon},x_j^{\epsilon}\in$  $V(\Gamma_0)$ ($i\neq j$). Then
\begin{alignat*}{1}
N_\Gamma(x_i^{\epsilon}) &= \{x_{i+\epsilon q_s}^{-\epsilon}\ |\ 1\leq s\leq k\},\\
N_\Gamma(x_j^{\epsilon}) &= \{x_{j+\epsilon q_t}^{-\epsilon}\ |\ 1\leq t\leq k\},
\end{alignat*}
so $|N_\Gamma(x_i^{\epsilon})\cap N_\Gamma(x_j^{\epsilon})|=1$ if and only if $i+\epsilon q_s\equiv j+\epsilon q_t$~mod~$n$ for some unique pair $q_s,q_t$ ($1\leq s,t\leq k$). By Theorem~\ref{thm:components} $i-j\equiv p\nu$~mod~$n$ for some $1\leq p<N$. Now (a),(b),(c) imply that
\begin{alignat*}{1}
\{ (q_i-q_j)~\mathrm{mod}~n\ |\ 1\leq i,j\leq k,\ i\neq j\} = \{\nu,2\nu,\ldots ,(N-1)\nu~\mathrm{mod}~n\}
\end{alignat*}
so there exists a unique pair $q_s,q_t\in\{q_1,\ldots ,q_k\}$ such that $\epsilon(q_t-q_s)\equiv p\nu$~mod~$n$; i.e.\,$\epsilon(q_t-q_s)\equiv i-j$~mod~$n$, or $i+\epsilon q_s\equiv j+\epsilon q_t$~mod~$n$, as required.
\end{proof}

Note that the `non-redundant' hypothesis cannot be removed from Theorem~\ref{thm:3kalphapositiveclassification}; a presentation that demonstrates this is given in Example~\ref{ex:specialcyclicpresentations}(e).

\begin{corollary}\label{cor:33alphaspecial}
Let $w=x_0x_{q_1}x_{q_1+q_2}$ where $0\leq q_1,q_2<n$ and suppose that $P_n(w)$ is irreducible. Then $P_n(w)$ is $(3,3,\nu)$-special if and only if either:
\begin{itemize}
\item[(a)] $n=7$, $\nu=1$, and $\{q_1,q_2,(-q_1-q_2)~\mathrm{mod}~7\} =\{1,2,4\}$ or $\{3,5,6\}$; or
\item[(b)] $n=21$, $\nu=3$, and $\{q_1,q_2,(-q_1-q_2)~\mathrm{mod}~21\} = \{1,4,16\}$, $\{2,8,11\}$, $\{5,17,20\}$, or $\{10,13,19\}$.
\end{itemize}
\end{corollary}

\begin{proof}
If $w$ is a proper power then $w=x_0^3$ and the star graph $\Gamma$ of $P_n(w)$ consists of vertices $x_i,x_i^{-1}$ and edges $x_i-x_i^{-1}$ ($0\leq i<n$). If $P_n(w)$ is redundant then $n=3$ and $q_1\equiv q_2 \equiv \pm 1$~mod~$3$, so $\Gamma$ consists of vertices $x_i,x_i^{-1}$ and edges $x_i-x_{i+1}^{-1}$ ($0\leq i<3$). Therefore, in each case, $P_n(w)$ is not $(3,3,\nu)$-special so we may assume that $w$ is not a proper power and $P_n(w)$ is non-redundant.

By Theorem~\ref{thm:3kalphapositiveclassification} the presentation $P_n(w)$ is $(3,3,\nu)$-special if and only if $n=7\nu$, $\nu=1$ or $3$, $\{q_1,q_2,-q_1-q_2\}$ is a perfect difference set and $q_1\equiv q_2 \equiv -(q_1+q_2)$~mod~$\nu$. The only possibilities for $q_1,q_2,-q_1-q_2$~mod~$n$ such that $\mathrm{gcd}(n,q_1,q_2)=1$ (for irreducibility) are those in the statement.
\end{proof}

The isomorphisms in~\cite[Lemma~2.1]{EdjvetWilliams} show that the presentations in Corollary~\ref{cor:33alphaspecial}~(a) each define the group $G_7(x_0x_1x_3)$ while those in part~(b) define the group $G_{21}(x_0x_1x_5)$. These groups are discussed in Example~\ref{ex:specialcyclicpresentations}.

Note that if $P_n(w)$ is a non-redundant $(3,k,\nu)$-special cyclic presentation where $\nu>1$ then $G$ is large by Corollary~\ref{cor:disconnectedlarge} and recall from Section~\ref{sec:specialpresentation} that a group defined by a $(3,k,1)$-special presentation is SQ-universal if and only if $k\neq 3$. Thus there is precisely one cyclically presented group defined by a non-redundant $(3,k,\nu)$-special cyclic presentation that is not SQ-universal, namely $G_7(x_0x_1x_3)$.

\section{Classification of non-redundant $(2,k,\nu)$-special cyclic presentations}\label{sec:nonredundant2kalpha}

In this section, we classify the non-redundant $(2,k,\nu)$-special cyclic presentations $P_n(w)$: in Theorem~\ref{thm:2kalphanonredpositive} we do this for positive words $w$, in Theorem~\ref{thm:2kalphanonredalternating} we do this for alternating words $w$, and in Theorem~\ref{thm:2kalphanonrednonpositivenonalternating} for words $w$ that are non-positive, non-negative, and non-alternating. In Corollaries~\ref{cor:24alphapositive},\ref{cor:24alphaalternating},\ref{cor:24alphanonposnonalt} we consider the Euclidean case, and classify the $(2,4,\nu)$-special cyclic presentations. Except in one case we show the groups defined by these $(2,4,\nu)$-special presentations are large; it follows that, with one possible exception, groups defined by non-redundant $(2,k,\nu)$-special cyclic presentations are SQ-universal.

\subsection{The positive case}\label{sec:nonredundant2kalphapositive}

\begin{maintheorem}\label{thm:2kalphanonredpositive}
Let $w$ be a positive word of length $k$ that is not a proper power and let $\mathcal{Q}$ be the multiset defined at~(\ref{eq:ABQ}). Suppose also that the cyclic presentation $P_n(w)$ is irreducible and non-redundant and let $\Gamma$ be the star graph of $P_n(w)$. Then $P_n(w)$ is $(2,k,\nu)$-special if and only if either
\begin{itemize}
\item[(a)] $k\geq 5$ is odd, $n=k$, $\nu=1$, $\mathcal{Q}= \{0,1,2,\ldots ,n-1\}$, in which case $\Gamma$ is the complete bipartite graph $K_{n,n}$; or
\item[(b)] $k\geq 4$ is even, $n=2k$, $\nu=2$, $\mathcal{Q}=\{1,3,\ldots ,n-1\}$, in which case $\Gamma$ is the disjoint union of two copies of $K_{n/2,n/2}$.
\end{itemize}
\end{maintheorem}

\begin{proof}
Suppose that $P_n(w)$ is an irreducible and non-redundant $(2,k,\nu)$-special cyclic presentation (so $k\geq 4$). By Corollary~\ref{cor:mkalphacharacterisation} $n=\nu k$ and the star graph $\Gamma$ of $P_n(w)$ has $\nu$ components each of which is a complete bipartite graph $K_{k,k}$. Then by Theorem~\ref{thm:components} $q\equiv q'$~mod~$\nu$ for all $q,q'\in \mathcal{Q}$. Note that for two elements $q,q'$ of the multiset $\mathcal{Q}$ we have $q\not\equiv q'$ mod $n$ for otherwise $\Gamma$ has a reduced closed path of length 2. Thus $\mathcal{Q}=\lbrace q_0, \nu +q_0, 2\nu +q_0, \ldots, (k-1)\nu +q_0\rbrace$ for some $q_0\in \mathcal{Q}$. Now $\mathrm{gcd}(q_0,\nu)$ divides $\mathrm{gcd}(n, q\ (q\in\mathcal{Q}))=1$ (since $P_n(w)$ is irreducible) so $\mathrm{gcd}(q_0,\nu)=1$.

Summing the elements of $\mathcal{Q}$ gives $k(k\nu-\nu+2q_0)/2$, so $k(k\nu-\nu+2q_0)/2\equiv 0$ mod $\nu k$. Hence $k(k\nu-\nu+2q_0)=2\nu k t$ for some integer $t$, and so $k=2t+1-(2q_0/\nu)$ so $\nu$ divides $2q_0$. But $\mathrm{gcd}(\nu,q_0)=1$ so $\nu$ divides $2$. Therefore $\nu\in \lbrace 1,2\rbrace$ and $k$ is odd if and only if $\nu=1$. If $\nu=1$ then $\mathcal{Q}=\lbrace q_0, 1 +q_0, 2 +q_0, \ldots, (n-1) +q_0\rbrace$, that is,  $\mathcal{Q}=\{0,1,2,\ldots ,n-1\}$, and if $\nu=2$ then $\mathcal{Q}=\lbrace q_0, 2+q_0, 4 +q_0, \ldots, 2(n/2-1) +q_0\rbrace$, that is, $\mathcal{Q}=\{1,3,\ldots ,n-1\}$.

For the converse suppose that (a) or (b) hold and let $\Lambda$ be a connected component of $\Gamma$. Then since no two distinct elements of $\lbrace 1, \nu+1, 2\nu+1, \ldots, (k-1)\nu+1  \rbrace$ are congruent mod~$n$ Theorem~\ref{thm:components} implies that the graph $\Lambda$ is $k$-regular (so each vertex has degree $\geq 3$) and contains no $2$-cycles and $\Gamma$ has $\nu$ components. When $\nu=1$ the set $\mathcal{Q}=\{0,1,\ldots ,n-1\}$ so every positive vertex is adjacent to every negative vertex and hence $\Gamma$ is the complete bipartite graph $K_{n,n}$. Suppose then that $\nu=2$. Then $k=n/2$ and $\mathcal{Q}=\{1,3,\ldots ,n-1\}$. By Theorem~\ref{thm:components} two positive vertices are in the same component if and only if their subscripts are of the same parity. Moreover, the set of neighbours of any even  (respectively odd) positive vertex consists of all odd (respectively even) negative vertices and hence each component is the complete bipartite graph $K_{n/2,n/2}$, as required.
\end{proof}

Corollary~\ref{cor:disconnectedlarge} implies that the groups $G_n(w)$ corresponding to part~(b) are large while those in part~(a) are SQ-universal (see Section~\ref{sec:specialpresentation}); Example~\ref{ex:2k1large} gives examples of large groups $G_n(w)$ that correspond to part~(a). Restricting to the Euclidean case, we now classify the non-redundant $(2,4,\nu)$-special cyclic presentations.

\begin{corollary}\label{cor:24alphapositive}
Let $P=P_n(x_0x_{q_1}x_{q_1+q_2}x_{q_1+q_2+q_3})$ be irreducible and non-redundant, where $0\leq q_1,q_2,q_3<n$ and let $G$ be the group defined by $P$. Then $P$ is $(2,4,\nu)$-special if and only if $n=8$, $\nu=2$, $q_1\in\{1,3,5,7\}$ and one of the following holds:
\begin{itemize}
  \item[(a)] ($q_2\equiv 3q_1$ and $q_3\equiv 5q_1$~mod~$n$) or ($q_2\equiv 7q_1$ and $q_3\equiv 5q_1$~mod~$n$) in which case $G\cong G_8(x_0x_1x_4x_1)$ which contains a subgroup isomorphic to $F_2\times F_2$;
  \item[(b)] ($q_2\equiv 5q_1$ and $q_3\equiv 3q_1$~mod~$n$) or ($q_2\equiv 7q_1$ and $q_3\equiv 3q_1$~mod~$n$) in which case $G\cong G_8(x_0x_1x_6x_1)$;
  \item[(c)] ($q_2\equiv 3q_1$ and $q_3\equiv 7q_1$~mod~$n$) or ($q_2\equiv 5q_1$ and $q_3\equiv 7q_1$~mod~$n$) in which case $G\cong G_8(x_0x_1x_4x_3)$.
\end{itemize}
\end{corollary}

\begin{proof}
Let $q_4=-(q_1+q_2+q_3)$~mod~$n$. By Theorem~\ref{thm:2kalphanonredpositive} we have $\nu=2$, $n=8$, $\{q_1,q_2,q_3,q_4\}=\{1,3,5,7\}$. Since $q_1\in\{1,3,5,7\}$ it has a multiplicative inverse mod~$n$, $q_1^{-1}\in \{1,3,5,7\}$. Define $Q_2=q_1^{-1}q_2$, $Q_3=q_1^{-1}q_3$, $Q_4=q_1^{-1}q_4$~mod~$n$. This implies that $(Q_2,Q_3)= (3,5),(3,7),(5,3),(5,7),(7,3),$ or $(7,5)$, and hence the stated values of $q_2,q_3$. Then by multiplying the subscripts of generators $x_i$ by $q_1^{-1}$ we have $G\cong G_n(x_0x_1x_{1+Q_2}x_{1+Q_2+Q_3})$. In the cases $(Q_2,Q_3)= (7,5),(7,3),(5,7)$ the isomorphism to the stated group comes about by subtracting $1$ from the subscripts of generators, multiplying them by 3 or 5 (mod~$8$), and cyclically permuting.

It remains to show that $G=G_8(x_0x_1x_4x_1)$ contains a subgroup isomorphic to $F_2 \times F_2$. A computation in GAP~\cite{GAP} shows that the subgroup of $G$ generated by the set of elements
\[\{x_0^{2}, x_1^{2}, x_3^{2}, x_2, x_6, x_4x_0^{-1}, x_4^{-1}x_0^{-1}, x_5x_1^{-1}, x_7x_3^{-1}, x_0x_1x_3^{-1}, x_0x_1^{-1}x_3^{-1}\}\]
is of index~4 and has a presentation whose set of generators contains elements $a,b,c,d$ and whose set of relators contains the commutators $[a,b],[b,c],[c,d],[d,a]$ and that the quotient obtained by killing all other generators has precisely these four generators and four relators, so defines $F_2\times F_2$. Therefore $a,b,c,d$ generate a subgroup of $G$ that is isomorphic to $F_2\times F_2$.
\end{proof}

Note that the groups in Corollary~\ref{cor:24alphapositive} are large by Corollary~\ref{cor:disconnectedlarge}. We have been unable to determine if the groups in parts~(b),(c) contain a subgroup isomorphic to $F_2\times F_2$.

\subsection{The alternating case}\label{sec:nonredundant2kalphaalternating}

\begin{maintheorem}\label{thm:2kalphanonredalternating}
Let $w$ be a cyclically reduced, alternating word of (even) length $k\geq 4$, let $\mathcal{A},\mathcal{B}$ be the multisets defined at~(\ref{eq:ABQ}) and suppose that $P_n(w)$ is irreducible and non-redundant. Then $P_n(w)$ is $(2,k,\nu)$-special if and only if $n=2k$, $\nu=2$, and $\mathcal{A}, \mathcal{B}$ each consist of $k/2$ odd integers and if $a\in \mathcal{A}$ then $n-a \not \in \mathcal{A}$ and if $b\in \mathcal{B}$ then $n-b \not \in \mathcal{B}$.
\end{maintheorem}

\begin{proof}
Let $\Gamma$ be the star graph of $P_n(w)$ and let $\Gamma^+,\Gamma^-$ be the induced subgraphs of $\Gamma$ whose vertex sets are the positive and negative vertices of $\Gamma$, respectively, and let $d_\mathcal{A}=\mathrm{gcd}(n, a\ (a\in\mathcal{A})),d_\mathcal{B}=\mathrm{gcd}(n, b\ (b\in\mathcal{B}))$.

Suppose first that the given conditions hold. These imply that either $1$ or $n-1\in\mathcal{A}$ so $d_\mathcal{A}=1$ and hence, by Theorem~\ref{thm:components}, $\Gamma^+$ is the circulant graph $\mathrm{circ}_n(\mathcal{A})$, which is the complete bipartite graph $K_{n/2,n/2}$ with vertex partition $\{x_0,x_2,\ldots,x_{n-2}\}\cup \{x_1,x_3,\ldots ,x_{n-1}\}$. Similarly $\Gamma^-$ is the circulant graph $\mathrm{circ}_n(\mathcal{B})$, which is the complete bipartite graph $K_{n/2,n/2}$ with vertex partition $\{x_0^{-1},x_2^{-1},\ldots,x_{n-2}^{-1}\}\cup \{x_1^{-1},x_3^{-1},\ldots ,x_{n-1}^{-1}\}$. Thus $P_n(w)$ is $(2,n/2,2)$-special.

Suppose then that $P_n(w)$ is $(2,k,\nu)$-special. Then by Corollary~\ref{cor:mkalphacharacterisation} each component of $\Gamma$ is the complete bipartite graph $K_{k,k}$. Then Theorem~\ref{thm:components} implies $d_\mathcal{A}=d_\mathcal{B}$, and since $P_n(w)$ is irreducible $1=\mathrm{gcd}(n, a\ (a \in\mathcal{A}), b\ (b \in\mathcal{B}))=(d_\mathcal{A},d_\mathcal{B})$ so $d_\mathcal{A}=d_\mathcal{B}=1$. Thus $\Gamma$ has 2 components so $\nu=2$, and each of these components must therefore be the complete bipartite graph $K_{n/2,n/2}$, and hence $k=n/2$. If $a\equiv \pm a'$~mod~$n$ for some $a,a'\in\mathcal{A}$ then $\Gamma$ contains a reduced closed path of length 2, contradicting the girth, so $a\not \equiv \pm a'$~mod~$n$ for all $a,a'\in \mathcal{A}$, and similarly $b\not \equiv \pm b'$~mod~$n$ for all $b,b'\in \mathcal{B}$. If $\mathcal{A}$ contains an even element, $a$, say then since either $1$ or $n-1\in \mathcal{A}$ the graph $\Gamma^+$ contains a closed path
\( x_0-x_1-x_2-\cdots -x_a-x_0\)
of length $a+1$, which is odd, a contradiction (since $\Gamma$ is bipartite). Therefore all elements of $\mathcal{A}$ are odd, and similarly all elements of $\mathcal{B}$ are odd, as required.
\end{proof}

Restricting to the Euclidean case, we now classify the non-redundant $(2,4,\nu)$-special cyclic presentations.

\begin{corollary}\label{cor:24alphaalternating}
Let $P=P_n(x_0x_{a_1}^{-1}x_{a_1+b_1}x_{a_1+b_1+a_2}^{-1})$ be an irreducible and non-redundant cyclic presentation, where $0\leq a_1,b_1,a_2<n$ and let $G$ be the group defined by $P$. Then $P$ is $(2,4,\nu)$-special if and only if $n=8$, $\nu=2$ and one of the following holds:
\begin{itemize}
  \item[(a)] ($a_2\equiv 5a_1$, $b_1\equiv 3a_1$~mod~$n$) or ($a_2\equiv 5a_1$, $b_1\equiv 7a_1$~mod~$n$), in which case $G\cong G_8(x_0x_1x_4x_1)$;
  \item[(b)] ($a_2\equiv 3a_1$, $b_1\equiv 5a_1$~mod~$n$) or ($a_2\equiv 3a_1$, $b_1\equiv 7a_1$~mod~$n$), in which case $G\cong G_8(x_0x_1x_6x_1)$;
  \item[(c)] ($a_2\equiv 3a_1$, $b_1\equiv a_1$~mod~$n$) or ($a_2\equiv 3a_1$, $b_1\equiv 3a_1$~mod~$n$), in which case $G\cong G_8(x_0x_1x_4x_3)$.
\end{itemize}
\end{corollary}

\begin{proof}
If $w=x_0x_{a_1}^{-1}x_{a_1+b_1}x_{a_1+b_1+a_2}^{-1}$ is a proper power then $w=(x_0x_{a_1}^{-1})^2$, in which case each vertex of the star graph of $P_n(w)$ has degree 2, so $P_n(w)$ is not $(2,4,\nu)$-special. Hence we may assume that $w$ is not a proper power. Let $b_2=-(a_1+b_1+a_2)$~mod~$n$. By Theorem~\ref{thm:2kalphanonredalternating} the presentation $P_n(w)$ is $(2,4,\nu)$-special if and only if $\nu=2$, $n=8$, $a_1,a_2,b_1,b_2$ are odd, and $a_1\not \equiv \pm a_2$, $b_1\not \equiv \pm b_2$~mod~$8$. Since $a_1\in\{1,3,5,7\}$ it has a multiplicative inverse mod~$n$, $a_1^{-1}\in \{1,3,5,7\}$. Define $B_1=a_1^{-1}b_1$, $A_2=a_1^{-1}a_2$, $B_2=a_1^{-1}b_2$~mod~$n$. By multiplying the subscripts of generators $x_i$ by $a_1^{-1}$ we have $G\cong G_n(x_0x_1^{-1}x_{1+B_1}x_{1+B_1+A_2}^{-1})$. Further $B_1,A_2,B_2 \in \{1,3,5,7\}$ and $A_2\not \equiv \pm 1$, $B_1\not \equiv \pm B_2$ and $1+B_1+A_2+B_2\equiv 0$~mod~$n$. This implies $(A_2,B_1)= (3,1),(3,3),(3,5),(3,7),(5,3),$ or $(5,7)$, and hence the stated values of $a_2,b_1$. In parts~(a),(b) isomorphisms $G_8(x_0x_1^{-1}x_6x_1^{-1})\cong G_8(x_0x_1x_6x_1)$, $G_8(x_0x_1^{-1}x_0x_3^{-1})\cong G_8(x_0x_1x_0x_3)$, $G_8(x_0x_1^{-1}x_4x_1^{-1})\cong G_8(x_0x_1x_4x_1)$, $G_8(x_0x_1^{-1}x_0x_5^{-1})\cong G_8(x_0x_1x_0x_5)$ are obtained by replacing each odd numbered generator by its inverse. Then $G_8(x_0x_1x_6x_1)\cong G_8(x_0x_1x_0x_3)$ and $G_8(x_0x_1x_4x_1)\cong G_8(x_0x_1x_0x_5)$ as in Corollary~\ref{cor:24alphapositive}. In part~(c) the isomorphism $G_8(x_0x_1^{-1}x_2x_5^{-1})\cong G_8(x_0x_1^{-1}x_4x_7^{-1})$ is obtained by inverting the relators and negating the subscripts. Then \linebreak $G_8(x_0x_1^{-1}x_2x_5^{-1})\cong G_8(x_0x_1x_6x_5)$ is obtained by inverting the odd numbered generators and interchanging $x_2$ and $x_6$ and interchanging $x_3$ and $x_7$ and then $G_8(x_0x_1x_6x_5)\cong G_8(x_0x_1x_4x_3)$ as in the proof of Corollary~\ref{cor:24alphapositive}.
\end{proof}

Note that the groups appearing in Corollary~\ref{cor:24alphaalternating} are the same as those in Corollary~\ref{cor:24alphapositive}, in particular they are large and $G_8(x_0x_1x_4x_1)$ contains a subgroup isomorphic to $F_2\times F_2$. We now show that all the groups defined by the $(2,k,\nu)$-special presentations in Theorem~\ref{thm:2kalphanonredalternating} are large.

\begin{corollary}\label{cor:alternatinglargeness}
Let $P_n(w)$ be a non-redundant and irreducible $(2,k,\nu)$-special cyclic presentation where $w$ is an alternating word of length at least 4. Then the cyclically presented group $G_n(w)$ defined by $P_n(w)$ is large.
\end{corollary}

\begin{proof}
Let $w=x_0x_{a_1}^{-1}x_{a_1+b_1}x_{a_1+b_1+a_2}^{-1}\ldots x_{\sum_{i=1}^{k/2-1}(a_i+b_i)}x_{\sum_{i=1}^{k/2-1}(a_i+b_i)+a_{k/2}}^{-1}$ and let $b_{k/2}=-\sum_{i=1}^{k/2-1}(a_i+b_i)-a_{k/2}$~mod~$n$. Then the shift extension $E=G_n(w)\rtimes_\theta \Z_n$ has a presentation \linebreak $E=\pres{x,t}{t^{2k},\prod_{i=1}^{k/2} (xt^{a_i}x^{-1}t^{b_i})}$ where each $a_i,b_i$ is odd, so $E$ maps onto the generalized triangle group $T=\pres{x,t}{x^7, t^{2}, (xtx^{-1}t)^{k/2}}$ which,  by~\cite{BMS}, is large if $k/2\geq 3$. Thus we may assume $k=4$, in which case $G$ is one of the groups in Corollary~\ref{cor:24alphaalternating}, which are large by Corollary~\ref{cor:disconnectedlarge}.
\end{proof}

\subsection{The non-positive, non-negative, non-alternating case}\label{sec:nonredundant2kalphanonpositivenonalternating}

\begin{maintheorem}\label{thm:2kalphanonrednonpositivenonalternating}
Let $w$ be a cyclically reduced word that is non-positive, non-negative, and non-alternating and not a proper power and let $\mathcal{A},\mathcal{B},\mathcal{Q},\mathcal{Q}^+,\mathcal{Q}^-$ be the multisets defined at~(\ref{eq:ABQ}) and suppose that $P_n(w)$ is irreducible and non-redundant. Then $P_n(w)$ is $(2,k,\nu)$-special if and only if the following hold:
\begin{itemize}
  \item[(a)] $n=\nu k$, and $k$ is divisible by $4$;
  \item[(b)] $\mathcal{A},\mathcal{B}\subset \{\nu,3\nu, \ldots , (k-1)\nu\}$ such that $|\mathcal{A}|=|\mathcal{B}|=k/4$ and if $a\in \mathcal{A}$ then $n-a\not \in \mathcal{A}$ and if $b\in \mathcal{B}$ then $n-b\not \in \mathcal{B}$;
  \item[(c)] $\mathcal{Q}^+\cap \mathcal{Q}^-=\emptyset$ and $\mathcal{Q}=\{s+\nu,s+3\nu,\ldots, s+(k-1)\nu\}$ for some $0\leq s<n$ such that $\mathrm{gcd}(s,\nu)=1$.
\end{itemize}
\end{maintheorem}

\begin{proof}
Suppose first that $P_n(w)$ is $(2,k,\nu)$-special. Then by Corollary~\ref{cor:mkalphacharacterisation} $n=k\nu$ and each component of the star graph $\Gamma$ of $P_n(w)$ is the complete bipartite graph $K_{k,k}$.

In the notation of Theorem~\ref{thm:components}, $\Gamma$ has $\nu$ isomorphic components $\Gamma_j$ ($0\leq j<\nu$) where, in particular, $\Gamma_0$ is the complete bipartite graph $K_{k,k}$ with vertex set $V(\Gamma_0)=V(\Gamma_0^+)\cup V(\Gamma_0^-)$ where $\Gamma_0^+$ and $\Gamma_0^-$ are the induced labelled subgraphs of $\Gamma$ with vertex sets
\begin{alignat*}{1}
V(\Gamma_0^+) &=\{ x_0,x_\nu,\ldots, x_{(k-1)\nu}\},\\
V(\Gamma_0^-) &=\{ x_{q_0}^{-1},x_{q_0+\nu}^{-1},\ldots, x_{q_0+(k-1)\nu}^{-1}\}
\end{alignat*}
for some $q_0\in \mathcal{Q}^+\cup \mathcal{Q}^-$. Suppose for contradiction that $\nu,n-\nu \not \in \mathcal{A}$ and $\nu,n-\nu \not \in \mathcal{B}$. Then for each $0\leq i<n$, vertices $x_i,x_{i+\nu}$ are not joined by an edge and vertices $x_i^{-1},x_{i+\nu}^{-1}$ are not joined by an edge. Therefore the vertices of $V(\Gamma_0^+)$ all belong to the same part of $\Gamma_0$ and the vertices of $V(\Gamma_0^-)$ all belong to the same part of $\Gamma_0$, and hence no two positive vertices of $\Gamma$ are joined by an edge and no two negative vertices of $\Gamma$ are joined by an edge, and so $\mathcal{A}=\mathcal{B}=\emptyset$, a contradiction, since $w$ is non-positive and non-negative. Therefore $\nu$ or $n-\nu\in \mathcal{A}$ and $\nu$ or $n-\nu\in \mathcal{B}$. Thus $\Gamma_0$ contains closed paths $x_0-x_\nu-\cdots-x_{(k-1)\nu}-x_0$ and $x_{q_0}^{-1}-x_{q_0+\nu}^{-1}-\cdots-x_{q_0+(k-1)\nu}^{-1}-x_{q_0}^{-1}$ of length $k$. Since $\Gamma_0$ is bipartite, $k$ is even. Therefore the vertices $x_\nu,x_{3\nu},\ldots, x_{(k-1)\nu}$ are precisely those positive vertices of $\Gamma_0$ that belong to a different part of $\Gamma_0$ to $x_0$ (and so are neighbours of $x_0$) and the vertices
$x_{q_0+\nu}^{-1},x_{q_0+3\nu}^{-1},\ldots, x_{q_0+(k-1)\nu}^{-1}$ are precisely those negative vertices of $\Gamma_0$ that belong to a different part of $\Gamma_0$ to $x_{q_0}^{-1}$ (and so are neighbours of $x_{q_0}^{-1}$) and hence $\mathcal{A},\mathcal{B}\subset \{\nu,3\nu,\ldots , (k-1)\nu\}$. Moreover, for each odd $t$ either $t\nu$ or $(k-t)\nu \in \mathcal{A}$ (resp.\,$\mathcal{B}$), and precisely one of $t\nu$ or $(k-t)\nu \in \mathcal{A}$ (resp.\,$\mathcal{B}$), for otherwise $\Gamma_0$ contains a reduced closed path of length~2, a contradiction. Therefore $|\mathcal{A}|=|\mathcal{B}|=k/4$, $k$ is divisible by 4, and (a),(b) are proved.

Since $x_{q_0}^{-1}$ and $x_{q_0+\nu}^{-1}$ belong in different parts of $\Gamma_0$, either $x_0$ and $x_{q_0}^{-1}$ belong to the same part of $\Gamma_0$ or $x_0$ and $x_{q_0+\nu}^{-1}$ belong to the same part of $\Gamma_0$. In the first case the negative neighbours of $x_0$ are $\{ x_{q_0+\nu}^{-1},x_{q_0+3\nu}^{-1},\ldots ,x_{q_0+(k-1)\nu}^{-1}\}$ so $\mathcal{Q}=\{s+\nu,s+3\nu,\ldots ,s+(k-1)\nu\}$ where $s=q_0$, and in the second case the negative neighbours of $x_0$ are $\{ x_{q_0}^{-1},x_{q_0+2\nu}^{-1},\ldots ,x_{q_0+(k-2)\nu}^{-1}\}$ so $\mathcal{Q}=\{s+\nu,s+3\nu,\ldots ,s+(k-1)\nu\}$ where $s=q_0+\nu$. Also $\mathcal{Q}^+\cap \mathcal{Q}^-=\emptyset$, for otherwise $\Gamma_0$ contains a reduced closed path of length~2, a contradiction. Finally, $\mathrm{gcd}(s,\nu)$ divides $\mathrm{gcd}(n,a\ (a\in\mathcal{A}),b\ (b\in\mathcal{B}), q\ (q \in \mathcal{Q}))=1$, since $P_n(w)$ is irreducible, so $\mathrm{gcd}(s,\nu)=1$, and (c) is proved.

Now suppose that the conditions in the statement hold. By Corollary~\ref{cor:mkalphacharacterisation}(b) we must show that the star graph $\Gamma$ of $P_n(w)$ consists of $\nu$ connected components, each of which is a complete bipartite graph $K_{k,k}$. By Theorem~\ref{thm:components} the graph $\Gamma$ has $\nu$ isomorphic components. Consider the component $\Gamma_0$ whose vertex set is $\{x_0,x_{\nu},\ldots, x_{(k-1)\nu}\}\cup \{x_s^{-1},x_{\nu+s}^{-1},\ldots, x_{(k-1)\nu+s}^{-1}\}$. The set of neighbours of $x_{j\nu}$ ($0\leq j<k$) is the set
\begin{alignat*}{1}
N_\Gamma(x_{j\nu})
&= \begin{cases}
\{ x_{\nu}, x_{3\nu}, \ldots , x_{(k-1)\nu}\} \cup \{ x_{\nu+s}^{-1}, x_{3\nu+s}^{-1}, \ldots , x_{(k-1)\nu+s}^{-1}\}
& \mathrm{if}~j~\mathrm{is~even},\\
\{ x_{0}, x_{2\nu}, \ldots , x_{(k-2)\nu}\} \cup \{ x_{s}^{-1}, x_{2\nu+s}^{-1}, \ldots , x_{(k-2)\nu+s}^{-1}\}
& \mathrm{if}~j~\mathrm{is~odd},
\end{cases}
\end{alignat*}
and so $\Gamma_0$ is bipartite with vertex partition
\begin{alignat*}{1}
&\ \{x_0,x_{2\nu},\ldots ,x_{(k-2)\nu}, x_{s}^{-1},x_{2\nu+s}^{-1},\ldots , x_{(k-2)\nu+s}^{-1}\}\\
&\dot{\cup} \{x_{\nu},x_{3\nu},\ldots ,x_{(k-1)\nu}, x_{\nu+s}^{-1},x_{3\nu+s}^{-1},\ldots , x_{(k-1)\nu+s}^{-1}\}.
\end{alignat*}
Further, for each $0\leq j<k$ the set of neighbours $N_\Gamma(x_{j\nu+s}^{-1})=N_\Gamma(x_{j\nu})$ so $\Gamma_0$ is a complete bipartite graph, as required.
\end{proof}

Note that with $\mathcal{A},\mathcal{B},\mathcal{Q}^+,\mathcal{Q}^-$ as defined at~(\ref{eq:ABQ}) we have
\begin{alignat}{1}
\sum_{a\in\mathcal{A}} a + \sum_{b\in\mathcal{B}} b + \sum_{q\in\mathcal{Q}^+} q - \sum_{q\in\mathcal{Q}^-} q \equiv 0~\mathrm{mod}~n.\label{eq:Qplussummation}
\end{alignat}

Recall (from Section~\ref{sec:specialpresentation}) that if $k>4$ then a group defined by a non-redundant $(2,k,\nu)$-special presentation is SQ-universal.
Restricting to the Euclidean case, we now classify the non-redundant $(2,4,\nu)$-special cyclic presentations $P_n(w)$. By cyclically permuting and taking the inverse of $w$, if necessary, we may assume that either $w=x_0x_p^{-1}x_q^{-1}x_r$ or $w=x_0x_p^{-1}x_qx_r$  for some $0\leq p,q,r<n$.

\begin{corollary}\label{cor:24alphanonposnonalt}
\begin{itemize}
  \item[(a)] Let $P=P_n(x_0x_p^{-1}x_q^{-1}x_r)$ be irreducible and non-redundant and let $G$ be the group defined by $P$. Then $P$ is $(2,4,\nu)$-special if and only if $n=4\nu$, and $(p,q,r)=(\nu,-s,\nu-s)$ or $(3\nu,-s,3\nu-s)$ (mod~$n$) where $\mathrm{gcd}(s,\nu)=1$, in which case $G$ contains a subgroup isomorphic to $F_2\times F_2$ and admits an epimorphism onto the free group of rank $\nu$ so is large if $\nu>1$. Moreover, if $\nu=1$ then one of the following holds:
      \begin{itemize}
        \item[(i)] $G\cong G_4(x_0x_1^{-1}x_0^{-1}x_1)\cong F_2\times F_2$; or
        \item[(ii)] $G\cong G_4(x_0x_1^{-2}x_2)$ which has $F_5\times F_5$ as an index 16 subgroup;
        \item[(iii)] $G\cong G_4(x_0x_1^{-1}x_2^{-1}x_3)$ which has  $F_3\times F_5$ as an index 8 subgroup.
      \end{itemize}

  \item[(b)] Let $P=P_n(x_0x_p^{-1}x_qx_r)$ be irreducible and let $G$ be the group defined by $P$. Then $P$ is $(2,4,\nu)$-special if and only if $\nu=1$, $n=4$, and one of the following holds:
      \begin{itemize}
        \item[(i)] $(p,q,r)\in \{(1,0,1), (3,0,3)\}$, in which case $G\cong G_4(x_0x_1^{-1}x_0x_1)$ and the (index 16) derived subgroup $G'\cong F_5\times F_5$;

        \item[(ii)] $(p,q,r)\in \{(1,0,3), (3,0,1)\}$, in which case $G\cong G_4(x_0x_1^{-1}x_0x_3)$ and $G$ has $F_5\times F_5$ as an index 16 subgroup;

        \item[(iii)] $(p,q,r)\in \{(1,2,0), (3,2,0), (1,2,2), (3,2,2)\}$, in which case $G\cong G_4(x_0^2x_1^{-1}x_2)$.
      \end{itemize}

\end{itemize}
\end{corollary}

\begin{proof}
(a) With the notation~(\ref{eq:ABQ}) we have $\mathcal{A}=\{p\}$, $\mathcal{B}=\{r-q\}$, $\mathcal{Q}^+=\{-r\}$, $\mathcal{Q}^-=\{p-q\}$. Then by Theorem~\ref{thm:2kalphanonrednonpositivenonalternating} the presentation $P$ is $(2,k,\nu)$-special if and only if $n=4\nu$, $p=r-q\in\{\nu,3\nu\}$, $\{p-q,-p-q\}= \{s+\nu, s+3\nu\}$ for some $\mathrm{gcd}(s,\nu)=1$. That is, $(p,q,r)=(\nu,-s,\nu-s)$, $(3\nu,-s,3\nu-s)$, $(\nu,2\nu-s,3\nu-s)$, $(3\nu,2\nu-s,\nu-s)$. Replacing $s$ by $2\nu+s$ in the last two cases transforms them to the first cases. If $\nu=1$ then then $\pm (p,q,r)=(1,0,1),(1,1,2),(1,2,3),(1,3,0)$ (mod~$n$) and $G$ is isomorphic to one of the stated groups. A computation in GAP reveals the subgroups claimed.
Suppose then that $\nu>1$. Then $G$ maps onto $G_\nu(x_0x_p^{-1}x_q^{-1}x_r)$, which is free of rank $\nu$, so $G$ is large. Moreover, the shift extension of $G$ is the group $E=\pres{x,t}{t^{4\nu}, xt^{\nu}x^{-1}t^{-(\nu+s)}x^{-1}t^\nu xt^{s-\nu}}$ (by replacing $s$ by $-s$, if necessary). Therefore $G$ is the kernel the epimorphism $\phi_0:E\rightarrow \pres{t}{t^8}$ given by $\phi_0(t)=t$, $\phi(x)=1$. On the other hand, the kernel of the epimorphism  $\phi_{-s}:E\rightarrow \pres{t}{t^8}$ given by $\phi_0(t)=t$, $\phi(x)=t^{-s}$  is the cyclically presented group $G_{4\nu}(y_0y_{\nu}^{-1}y_0^{-1}y_\nu)$, where $y_i=t^{i}xt^{-i+s}$, which is isomorphic to the free product of $\nu$ copies of $G_4(y_0y_1^{-1}y_0^{-1}y_1)\cong F_2\times F_2$. In particular, the subgroup of $\ker (\phi_{-s})$ generated by $y_0,y_\nu,y_{2\nu},y_{3\nu}$ is the group
\begin{alignat*}{1}
H&=\pres{y_0,y_\nu,y_{2\nu},y_{3\nu}}{y_0y_\nu^{-1}y_{0}^{-1}y_{\nu},y_\nu y_{2\nu}^{-1}y_{\nu}^{-1}y_{2\nu},y_{2\nu}y_{3\nu}^{-1}y_{2\nu}^{-1}y_{3\nu},y_{3\nu}y_0^{-1}y_{3\nu}^{-1}y_{0}}\\
&=\pres{y_0,y_{2\mu}}{} \times \pres{y_{\nu},y_{3\mu}}{}\cong F_2\times F_2.
\end{alignat*}
Therefore the subgroup $K$ of $H$ generated by $y_0^n,y_\nu^n,y_{2\nu}^n,y_{3\nu}^n$ is isomorphic to $F_2\times F_2$, and since $\phi_0(y_i^n)=1$, we have $K$ is a subgroup of $\ker (\phi_0)=G$, as required.

(b) By Theorem~\ref{thm:2kalphanonrednonpositivenonalternating} $P$ is $(2,k,\nu)$-special if and only if $n=4\nu$, $\mathcal{A}=\{p\}$, $\mathcal{B}=\{q-p\}$, $\mathcal{Q}^-=\emptyset$, $\mathcal{Q}^+=\{r-q,-r \}=\{s+\nu,s+3\nu\}$, for some $0\leq s<n$ and $\mathrm{gcd}(s,\nu)=1$, and $\mathcal{A} \cup \mathcal{B}\subset\{\nu,3\nu\}$, so $p\equiv q-p$ or $p\equiv-(q-p)$~mod~$n$.

Suppose $p\equiv -(q-p)$~mod~$n$; then~(\ref{eq:Qplussummation}) implies $2s\equiv 0$~mod~$n$ so $s\equiv 0$ or $2\nu$~mod~$n$, and then $\mathrm{gcd}(s,\nu)=1$ implies $\nu=1$, so $n=4$ and $s\equiv 0$ or $s\equiv 2$~mod~$4$. Then (mod~$4$) $p\in\{1,3\}$, $\{r-q,-r\}=\{1,3\}$, which has solutions $(p,q,r)=(1,0,1),(1,0,3),(3,0,1),(3,0,3)$, as in parts (i),(ii). Computations in GAP reveal the $F_5\times F_5$ subgroups.
Suppose $p\equiv q-p$~mod~$n$; then~(\ref{eq:Qplussummation}) implies $2(s+\nu)\equiv 0$~mod~$n$ so $s\equiv \nu$ or $3\nu$~mod~$n$, and then $\mathrm{gcd}(s,\nu)=1$ implies that $\nu=1$, so $n=4$ and $s\equiv 1$ or $s\equiv 3$~mod~$4$. Then (mod~$4$) $p=q-p\in \{1,3\}$ and $\{r-q,-r\}=\{0,2\}$, the solutions of which are $(p,q,r)=(1,2,2),(1,2,0),(3,2,2),(3,2,0)$, in which case $G\cong G_4(x_0^2x_1^{-1}x_2)$, as in part~(iii).
\end{proof}

The argument in the proof above for the existence of the $F_2\times F_2$ subgroup in the groups arising in part~(a) has its origins in~\cite[Example~3(b)]{BW2}.

We now determine which groups from Corollary~\ref{cor:24alphanonposnonalt} are Burger-Mozes groups, as defined in~\cite{KimberleyRobertson}, whose notation $2\times 2.j$ we use. Since these groups have deficiency at least zero, if they are Burger-Mozes groups then they have degree (4,4), by~\cite[Proposition~4.26]{RattaggiThesis}; the Burger-Mozes groups of degree~(4,4) are classified in the Table in~\cite{KimberleyRobertson}.
Consider first the case $\nu>1$, so $G$ is a group from part~(a). If $\nu>3$ then $G$ maps onto the free group of rank 4 so $G^\mathrm{ab}$ maps onto $\Z^4$, but the only group from the Table in~\cite{KimberleyRobertson} whose abelianisation maps onto $\Z^4$ is the group $2\times 2.41\cong F_2\times F_2$, which does not map onto the free group of rank 4. When $\nu=3$, a computation in GAP shows that the abelianisation $G^\mathrm{ab}$ is distinct from the abelianisations of the groups in~\cite{KimberleyRobertson}. When $\nu=2$, a comparison of abelianisations shows that if $G$ is isomorphic to a group $H$ from the Table in~\cite{KimberleyRobertson} then $H$ is the group $2\times 2.32$; but then $G$ can be distinguished from $H$ by comparing abelianisations of index~2 subgroups. The groups in parts~(a)(ii) and (b)(iii) can be distinguished from the groups in the Table in~\cite{KimberleyRobertson} by considering their abelianisations, or the abelianisations of their index 2 subgroups. The groups in parts (a)(i),(a)(iii),(b)(i),(b)(ii) are the Burger-Mozes groups $2\times 2.41\cong F_2\times F_2$, $2\times 2.51$, $2\times 2.12$, $2\times 2.36$, respectively of~\cite{KimberleyRobertson}. Thus, if a Burger-Mozes group is defined by a $(2,4,\nu)$-special presentation, then it is one of these four groups.

The results of Sections~\ref{sec:nonredundant3kalpha},\ref{sec:nonredundant2kalphapositive},\ref{sec:nonredundant2kalphaalternating} and Corollary~\ref{cor:24alphanonposnonalt} show that there are at most two groups defined by non-redundant $(m,k,\nu)$-special cyclic presentations that are not SQ-universal, namely $G_7(x_0x_1x_3)$ (which is not SQ-universal) and $G\cong G_4(x_0x_1^{-1}x_2x_0)$ (which remains unresolved).

\section*{Acknowledgement}

The authors thank Jim Howie for comments on a draft of this article.

  \textsc{Department of Mathematical Sciences, University of Essex, Wivenhoe Park, Colchester, Essex CO4 3SQ, UK.}\par\nopagebreak
  \textit{E-mail address}, \texttt{Ihechukwu.Chinyere@essex.ac.uk}

  \textsc{Department of Mathematical Sciences, University of Essex, Wivenhoe Park, Colchester, Essex CO4 3SQ, UK.}\par\nopagebreak
  \textit{E-mail address}, \texttt{Gerald.Williams@essex.ac.uk}

\begin{thebibliography}{10}

\bibitem{BallmannBrin94}
W.~Ballmann and M.~Brin.
\newblock Polygonal complexes and combinatorial group theory.
\newblock {\em Geom. Dedicata}, 50(2):165--191, 1994.

\bibitem{BallmannBrinOrbihedra}
Werner Ballmann and Michael Brin.
\newblock Orbihedra of nonpositive curvature.
\newblock {\em Publications Math\'ematiques de l'IH\'ES}, 82:169--209, 1995.

\bibitem{BBPV}
Nathan Barker, Nigel Boston, Norbert Peyerimhoff, and Alina Vdovina.
\newblock An infinite family of 2-groups with mixed {B}eauville structures.
\newblock {\em Int. Math. Res. Not. IMRN}, 2015(11):3598--3618, 2015.

\bibitem{Barre}
Sylvain Barr\'{e}.
\newblock Poly\`edres finis de dimension {$2$} \`a courbure {$\leq0$} et de
  rang {$2$}.
\newblock {\em Ann. Inst. Fourier (Grenoble)}, 45(4):1037--1059, 1995.

\bibitem{BMS}
Gilbert Baumslag, John~W. Morgan, and Peter~B. Shalen.
\newblock Generalized triangle groups.
\newblock {\em Math. Proc. Cambridge Philos. Soc.}, 102(1):25--31, 1987.

\bibitem{Bigdely}
Hadi Bigdely.
\newblock {\em Subgroup theorems in relatively hyperbolic groups and
  small-cancellation theory}.
\newblock PhD thesis, Department of Mathematics and Statistics, McGill
  University, 2013.

\bibitem{BigdelyWise}
Hadi Bigdely and Daniel~T. Wise.
\newblock Quasiconvexity and relatively hyperbolic groups that split.
\newblock {\em Michigan Math. J.}, 62(2):387--406, 2013.

\bibitem{Biggs}
Norman~L. Biggs.
\newblock {\em Discrete mathematics}.
\newblock Oxford Science Publications. The Clarendon Press, Oxford University
  Press, New York, second edition, 1989.

\bibitem{BoeschTindell}
F.~Boesch and R.~Tindell.
\newblock Circulants and their connectivities.
\newblock {\em J. Graph Theory}, 8(4):487--499, 1984.

\bibitem{BogleyShift}
William~A. Bogley.
\newblock On shift dynamics for cyclically presented groups.
\newblock {\em J. Algebra}, 418:154--173, 2014.

\bibitem{BW2}
William~A. Bogley and Gerald Williams.
\newblock Coherence, subgroup separability, and metacyclic structures for a
  class of cyclically presented groups.
\newblock {\em J. Algebra}, 480:266--297, 2017.

\bibitem{Bridson}
Martin~R. Bridson.
\newblock On the existence of flat planes in spaces of nonpositive curvature.
\newblock {\em Proc. Amer. Math. Soc.}, 123(1):223--235, 1995.

\bibitem{CarboneKangaslampiVdovina}
Lisa Carbone, Riikka Kangaslampi, and Alina Vdovina.
\newblock Groups acting simply transitively on vertex sets of hyperbolic
  triangular buildings.
\newblock {\em LMS J. Comput. Math.}, 15:101--112, 2012.

\bibitem{CMSZ1}
Donald~I. Cartwright, Anna~Maria Mantero, Tim Steger, and Anna Zappa.
\newblock Groups acting simply transitively on the vertices of a building of
  type {$\tilde{A}_2$}. {I}.
\newblock {\em Geom. Dedicata}, 47(2):143--166, 1993.

\bibitem{CMSZ2}
Donald~I. Cartwright, Anna~Maria Mantero, Tim Steger, and Anna Zappa.
\newblock Groups acting simply transitively on the vertices of a building of
  type {$\tilde{A}_2$}. {II}. {T}he cases {$q=2$} and {$q=3$}.
\newblock {\em Geom. Dedicata}, 47(2):167--223, 1993.

\bibitem{ChinyereWilliamsT6}
Ihechukwu Chinyere and Gerald Williams.
\newblock Hyperbolicity of {$T(6)$} cyclically presented groups.
\newblock {\em arXiv preprint arXiv:2006.09018}, 2020.

\bibitem{Collins73}
Donald~J. Collins.
\newblock Free subgroups of small cancellation groups.
\newblock {\em Proc. London Math. Soc. (3)}, 26:193--206, 1973.

\bibitem{Delzant}
Thomas Delzant.
\newblock Sous-groupes distingu\'{e}s et quotients des groupes hyperboliques.
\newblock {\em Duke Math. J.}, 83(3):661--682, 1996.

\bibitem{EdjvetBalanced}
M.~Edjvet.
\newblock Groups with balanced presentations.
\newblock {\em Arch. Math. (Basel)}, 42(4):311--313, 1984.

\bibitem{EdjvetIrreducibleCyclicPresentations}
Martin Edjvet.
\newblock On irreducible cyclic presentations.
\newblock {\em J. Group Theory}, 6(2):261--270, 2003.

\bibitem{EdjvetHowie}
Martin Edjvet and James Howie.
\newblock Star graphs, projective planes and free subgroups in small
  cancellation groups.
\newblock {\em Proc. London Math. Soc. (3)}, 57(2):301--328, 1988.

\bibitem{EdjvetVdovina}
Martin Edjvet and Alina Vdovina.
\newblock On the {SQ}-universality of groups with special presentations.
\newblock {\em J. Group Theory}, 13(6):923--931, 2010.

\bibitem{EdjvetWilliams}
Martin Edjvet and Gerald Williams.
\newblock The cyclically presented groups with relators {$x_ix_{i+k}x_{i+l}$}.
\newblock {\em Groups Geom. Dyn.}, 4(4):759--775, 2010.

\bibitem{FeitHigman}
Walter Feit and Graham Higman.
\newblock The nonexistence of certain generalized polygons.
\newblock {\em J. Algebra}, 1:114--131, 1964.

\bibitem{GaboriauPaulin}
Damien Gaboriau and Fr\'{e}d\'{e}ric Paulin.
\newblock Sur les immeubles hyperboliques.
\newblock {\em Geom. Dedicata}, 88(1-3):153--197, 2001.

\bibitem{GAP}
The GAP~Group.
\newblock {\em {GAP -- Groups, Algorithms, and Programming, Version 4.10.2}},
  2019.

\bibitem{GerstenShortI}
S.~M. Gersten and H.~B. Short.
\newblock Small cancellation theory and automatic groups.
\newblock {\em Invent. Math.}, 102(2):305--334, 1990.

\bibitem{AlexandrovBusemann}
\'{E}. Ghys and P.~de~la Harpe, editors.
\newblock {\em Sur les groupes hyperboliques d'apr\`es {M}ikhael {G}romov},
  volume~83 of {\em Progress in Mathematics}.
\newblock Birkh\"{a}user Boston, Inc., Boston, MA, 1990.
\newblock Papers from the Swiss Seminar on Hyperbolic Groups held in Bern,
  1988.

\bibitem{Grunewald}
Fritz~J. Grunewald.
\newblock On some groups which cannot be finitely presented.
\newblock {\em J. London Math. Soc. (2)}, 17(3):427--436, 1978.

\bibitem{Heuberger}
Clemens Heuberger.
\newblock On planarity and colorability of circulant graphs.
\newblock {\em Discrete Math.}, 268(1-3):153--169, 2003.

\bibitem{HillPrideVella}
Patricia Hill, Stephen~J. Pride, and Alfred~D. Vella.
\newblock On the {$T(q)$}-conditions of small cancellation theory.
\newblock {\em Israel J. Math.}, 52(4):293--304, 1985.

\bibitem{Howie89}
James Howie.
\newblock On the {${\rm SQ}$}-universality of {$T(6)$}-groups.
\newblock {\em Forum Math.}, 1(3):251--272, 1989.

\bibitem{HowieWilliams}
James Howie and Gerald Williams.
\newblock Tadpole labelled oriented graph groups and cyclically presented
  groups.
\newblock {\em J. Algebra}, 371:521--535, 2012.

\bibitem{HowieWilliamsPlanarity}
James Howie and Gerald Williams.
\newblock Planar {W}hitehead graphs with cyclic symmetry arising from the study
  of {D}unwoody manifolds.
\newblock {\em Discrete Math.}, 343(12):112096, 18, 2020.

\bibitem{Johnson97}
D.~L. Johnson.
\newblock {\em Presentations of groups}, volume~15 of {\em London Mathematical
  Society Student Texts}.
\newblock Cambridge University Press, Cambridge, second edition, 1997.

\bibitem{JWW}
D.~L. Johnson, J.~W. Wamsley, and D.~Wright.
\newblock The {F}ibonacci groups.
\newblock {\em Proc. London Math. Soc. (3)}, 29:577--592, 1974.

\bibitem{KangaslampiVdovina06}
Riikka Kangaslampi and Alina Vdovina.
\newblock Triangular hyperbolic buildings.
\newblock {\em C. R. Math. Acad. Sci. Paris}, 342(2):125--128, 2006.

\bibitem{KangaslampiVdovinaIJAC}
Riikka Kangaslampi and Alina Vdovina.
\newblock Cocompact actions on hyperbolic buildings.
\newblock {\em Internat. J. Algebra Comput.}, 20(4):591--603, 2010.

\bibitem{KangaslampiVdovina17}
Riikka Kangaslampi and Alina Vdovina.
\newblock Hyperbolic triangular buildings without periodic planes of genus 2.
\newblock {\em Exp. Math.}, 26(1):54--61, 2017.

\bibitem{Kantor}
William~M. Kantor.
\newblock Generalized polygons, {SCAB}s and {GAB}s.
\newblock In {\em Buildings and the geometry of diagrams ({C}omo, 1984)},
  volume 1181 of {\em Lecture Notes in Math.}, pages 79--158. Springer, Berlin,
  1986.

\bibitem{KimberleyRobertson}
Jason~S. Kimberley and Guyan Robertson.
\newblock Groups acting on products of trees, tiling systems and analytic
  {$K$}-theory.
\newblock {\em New York J. Math.}, 8:111--131, 2002.

\bibitem{LyndonSchupp}
Roger~C. Lyndon and Paul~E. Schupp.
\newblock {\em Combinatorial group theory}.
\newblock Classics in Mathematics. Springer-Verlag, Berlin, 2001.
\newblock Reprint of the 1977 edition.

\bibitem{Malcev}
A.I. {Mal'cev}.
\newblock On homomorphisms to finite groups.
\newblock {\em American Mathematical Society Translations Series 2},
  119:67--79, 1983.

\bibitem{Mihailova}
K.~A. Miha\u{\i}lova.
\newblock The occurrence problem for direct products of groups.
\newblock {\em Dokl. Akad. Nauk SSSR}, 119:1103--1105, 1958.

\bibitem{MohamedWilliams}
Esamaldeen Mohamed and Gerald Williams.
\newblock An investigation into the cyclically presented groups with length
  three positive relators.
\newblock {\em Experimental Mathematics}, pages 1--15, 2019.

\bibitem{OlshanskiiSQ}
A.~Yu. Olshanski\u{\i}.
\newblock {${\rm SQ}$}-universality of hyperbolic groups.
\newblock {\em Mat. Sb.}, 186(8):119--132, 1995.

\bibitem{PrideLargeness}
Stephen~J. Pride.
\newblock The concept of ``largeness'' in group theory.
\newblock In {\em Word problems, {II} ({C}onf. on {D}ecision {P}roblems in
  {A}lgebra, {O}xford, 1976)}, volume~95 of {\em Stud. Logic Foundations
  Math.}, pages 299--335. North-Holland, Amsterdam-New York, 1980.

\bibitem{PrideStarComplexes}
Stephen~J. Pride.
\newblock Star-complexes, and the dependence problems for hyperbolic complexes.
\newblock {\em Glasgow Math. J.}, 30(2):155--170, 1988.

\bibitem{Radu}
Nicolas Radu.
\newblock A lattice in a residually non-{D}esarguesian
  {$\tilde{A}_2$}-building.
\newblock {\em Bull. Lond. Math. Soc.}, 49(2):274--290, 2017.

\bibitem{Rattaggi07}
Diego Rattaggi.
\newblock A finitely presented torsion-free simple group.
\newblock {\em J. Group Theory}, 10(3):363--371, 2007.

\bibitem{RattaggiThesis}
Diego~Attilio Rattaggi.
\newblock {\em Computations in groups acting on a product of trees: {N}ormal
  subgroup structures and quaternion lattices}.
\newblock ProQuest LLC, Ann Arbor, MI, 2004.
\newblock Thesis (Dr.sc.math.)--Eidgenoessische Technische Hochschule Zuerich
  (Switzerland).

\bibitem{SingerFiniteProjGeom}
James Singer.
\newblock A theorem in finite projective geometry and some applications to
  number theory.
\newblock {\em Trans. Amer. Math. Soc.}, 43(3):377--385, 1938.

\bibitem{Stevenson}
Fredrick~W. Stevenson.
\newblock {\em Projective planes}.
\newblock W. H. Freeman and Co., San Francisco, Calif., 1972.

\bibitem{Maldeghem}
Hendrik Van~Maldeghem.
\newblock {\em Generalized polygons}.
\newblock Modern Birkh\"{a}user Classics. Birkh\"{a}user/Springer Basel AG,
  Basel, 1998.
\newblock 2011 reprint of the 1998 original.

\bibitem{Vdovina02}
Alina Vdovina.
\newblock Combinatorial structure of some hyperbolic buildings.
\newblock {\em Math. Z.}, 241(3):471--478, 2002.

\bibitem{Vdovina05}
Alina Vdovina.
\newblock Groups, periodic planes and hyperbolic buildings.
\newblock {\em J. Group Theory}, 8(6):755--765, 2005.

\bibitem{Wong}
Pak~Ken Wong.
\newblock Cages -- a survey.
\newblock {\em J. Graph Theory}, 6(1):1--22, 1982.

\end{thebibliography}
\end{document}